\documentclass[12pt]{amsart}
\usepackage{amsmath, amsthm, amssymb}
\usepackage[mathscr]{eucal}

\newcommand{\A}{\mathbb{A}}

\newcommand{\Q}{\mathbb{Q}}
\newcommand{\R}{\mathbb{R}}

\newcommand{\Z}{\mathbb{Z}}

\newcommand{\cExt}{\mathcal{E}xt}
\newcommand{\cF}{\mathcal{F}}
\newcommand{\cG}{\mathcal{G}}
\newcommand{\cH}{\mathcal{H}}
\newcommand{\cHom}{\mathcal{H}om}

\newcommand{\cO}{\mathcal{O}}

\newcommand{\cU}{\mathcal{U}}
\newcommand{\cV}{\mathcal{V}}

\newcommand{\cX}{\mathcal{X}}
\newcommand{\cY}{\mathcal{Y}}
\newcommand{\cZ}{\mathcal{Z}}

\newcommand{\scC}{\mathscr{C}}

\newcommand{\bA}{\mathbf{A}}
\newcommand{\bC}{\mathbf{C}}
\newcommand{\bF}{\mathbf{F}}
\newcommand{\bG}{\mathbf{G}}
\newcommand{\bH}{\mathbf{H}}
\newcommand{\bL}{\mathbf{L}}
\newcommand{\bM}{\mathbf{M}}
\newcommand{\bN}{\mathbf{N}}
\newcommand{\bP}{\mathbf{P}}
\newcommand{\bQ}{\mathbf{Q}}

\newcommand{\bT}{\mathbf{T}}

\newcommand{\gd}{\mathfrak{d}}
\def\gg{\mathfrak{g}}

\newcommand{\gp}{\mathfrak{p}}

\newcommand{\gU}{\mathfrak{U}}
\newcommand{\gV}{\mathfrak{V}}
\newcommand{\gW}{\mathfrak{W}}

\newcommand{\ab}{\mathrm{ab}}

\newcommand{\BS}{\mathrm{B.S.}}
\newcommand{\cd}{\widetilde{\mathrm{cd}}}
\newcommand{\cl}{\mathrm{cl.}}

\newcommand{\coker}{\mathrm{coker}\ }
\newcommand{\Cong}{\mathrm{Cong}}
\newcommand{\cpct}{\mathrm{c}}
\newcommand{\cts}{\mathrm{cts}}

\newcommand{\Ext}{\mathrm{Ext}}

\newcommand{\Group}{\mathrm{Group}}
\newcommand{\GL}{\mathrm{GL}}
\newcommand{\tH}{\tilde H}
\newcommand{\Hom}{\mathrm{Hom}}
\newcommand{\im}{\mathrm{Im}}
\newcommand{\ind}{\mathrm{ind}}
\newcommand{\hind}{\widehat{\mathrm{ind}}}

\newcommand{\la}{k_{\gp}-\mathrm{loc.an.}}

\newcommand{\limd}[1]{
	\begin{array}{c}
		\lim\\
		\stackrel{\textstyle \rightarrow}{\scriptstyle #1}
	\end{array}
}
\newcommand{\limp}[1]{
	\begin{array}{c}
		\lim\\
		\stackrel{\textstyle \leftarrow}{\scriptstyle #1}
	\end{array}
}
\newcommand{\limprojs}{\limp{s}}
\newcommand{\limdr}{\limd{r}}

\newcommand{\limdKp}{\limd{K_{\gp}}}

\newcommand{\modG}{\mathbf{mod}_{G}}
\newcommand{\modaG}{\mathbf{mod}_{G}^{\mathbf{adm}}}

\newcommand{\rank}{\mathrm{rank}}

\newcommand{\rel}{\mathrm{rel.}}

\newcommand{\Rest}{\mathrm{Rest}}

\newcommand{\sh}{\mathbf{sheaf}}

\newcommand{\sing}{\mathrm{sing}}
\newcommand{\SL}{\mathrm{SL}}

\newcommand{\sm}{\mathrm{smooth}}

\newcommand{\smooth}{\mathrm{smooth}}
\newcommand{\SmoG}{\mathbf{smo}_{G}}

\newcommand{\smoothind}{\mathrm{smoothind}}

\newcommand{\st}{\mathrm{st}}

\newcommand{\univ}{\mathrm{univ}}
\newcommand{\verl}{\mathrm{Verl}}

\newcommand{\uW}{\underline{W}}
\newcommand{\Zps}{\Z/p^{s}}
\newcommand{\uZps}{\underline{\Zps}}

\newtheorem{theorem}{Theorem}

\newtheorem{lemma}{Lemma}
\newtheorem{corollary}{Corollary}
\newtheorem{proposition}{Proposition}
\newtheorem{conjecture}{Conjecture}

\theoremstyle{definition}

\theoremstyle{remark}
\newtheorem*{remark}{Remark}

\begin{document}

\title{On Emerton's $p$-adic Banach spaces}
\author{Richard Hill}
\date{3rd August 2007}
\begin{abstract}
The purpose of the current paper is to introduce some new methods
 for studying the $p$-adic Banach spaces introduced by Emerton \cite{emerton}.
We first relate these spaces to more familiar sheaf cohomology groups.
As an application, we obtain a more general version of Emerton's spectral sequence.
We also calculate the spaces in some easy cases.
As a consequence, we obtain a number of vanishing theorems.
\end{abstract}
\address{Department of Mathematics, University College London.}
\email{rih@math.ucl.ac.uk}

\maketitle

\tableofcontents

\section{Introduction and Statements of Results}

\subsection{Cohomology of arithmetic quotients}

Let $\bG$ be a linear algebraic group over a number field $k$.
We choose a maximal compact subgroup $K_{\infty}\subset \bG(k_{\infty})$
and let $K_{\infty}^{\circ}$ be the identity component in $K_{\infty}$.
This paper is concerned with the cohomology of the
following ``arithmetic quotients'':
$$
	Y(K_{f})
	=
	\bG(k) \backslash \bG(\A) / K_{\infty}^{\circ}K_{f}.
$$
Fix once and for all a finite prime $\gp$ of $k$ and let $p$ be the rational prime
below $\gp$.
By a ``tame level'' we shall mean a compact open subgroup
 $K^{\gp}$ of $\bG(\A_{f}^{\gp})$.
For a field $E$ of characteristic zero,
 the level $K^{\gp}$ Hecke algebra is defined by
$$
	\cH(K^{\gp},E)
	=
	\{f:K^{\gp}\backslash \bG(\A_{f}^{\gp}) / K^{\gp} \to E :
	\hbox{$f$ has compact support} \}.
$$
Given a finite dimensional algebraic representation $W$ of $\bG$ over
an extension $E/k_{\gp}$, one defines a local system $\cV_{W}$ on
each arithmetic quotient $Y(K_{f})$.
We define the classical cohomology groups of tame level $K^{\gp}$ as follows:
$$
	H^{\bullet}_{\cl,\ast}(K^{\gp},W)
	=
	\limdKp
	H^{\bullet}_{\ast}(Y(K^{\gp}K_{\gp}),\cV_{W}),
$$
where $K_{\gp}$ ranges over the compact open subgroups of $\bG(k_{\gp})$.
The symbol ``$\ast$'' denotes either the empty symbol, meaning usual cohomology,
or ``$\cpct$'', meaning compactly supported cohomology.
There is a smooth action of $\bG(k_{\gp})$ on $H^{\bullet}_{\cl}(K^{\gp},W)$,
and there is also an action of the level $K^{\gp}$ Hecke algebra $\cH(K^{\gp},E)$.
The systems of Hecke eigenvalues arising in $H^{\bullet}_{\cl,\ast}(K^{\gp},W)$
are of considerable interest in number theory.

Traditionally, the classical cohomology groups have been studied using the theory
of automorphic representations (for example in \cite{borel-wallach}).
Recently, Emerton introduced a new method
for studying the classical cohomology groups.
Instead of studying the space of automorphic forms on $\bG$,
Emerton introduced a collection of $p$-adic Banach spaces, from which
the classical cohomology can be recovered.
Emerton's spaces are defined as follows:
$$
	\tH^{\bullet}_{\ast}(K^{\gp},\Z_{p})
	=
	\limprojs\limdKp H^{\bullet}_{\ast}(Y(K^{\gp}K_{\gp}),\Z/p^{s}).
$$
For a finite extension $E/\Q_{p}$ we also define
$$
	\tH^{\bullet}_{\ast}(K^{\gp},E)
	=
	\tH^{\bullet}_{\ast}(K^{\gp},\Z_{p})
	\otimes_{\Z_{p}}
	E.
$$
The space $\tH^{\bullet}_{\ast}(K^{\gp},E)$ has the structure
of a Banach space over $E$, where the unit ball is defined to be the
$\cO_{E}$-span of the image of $\tH^{\bullet}_{\ast}(K^{\gp},\Z_{p})$.
It is also convenient to consider the direct limit of these groups over the tame levels:
$$
	\tH^{\bullet}_{\ast}(\bG,-)
	=
	\limd{K^{\gp}}\tH^{\bullet}_{\ast}(K^{\gp},-).
$$
We have the following actions on these spaces:
\begin{enumerate}
	\item
	The group $\bG(\A_{f}^{\gp})$ acts smoothly on $\tH^{\bullet}_{\ast}(\bG,\Z_{p})$;
	the subspace $\tH^{\bullet}_{\ast}(K^{\gp},E)$ may be recovered as the
	$K^{\gp}$-invariants in $\tH^{\bullet}_{\ast}(\bG,E)$.
	\item
	The Hecke algebra $\cH(K^{\gp},E)$
	acts on $\tH^{\bullet}_{\ast}(K^{\gp},E)$.
	\item
	The group $\bG(k_{\gp})$ acts continuously,
	but not usually smoothly on the spaces
	$\tH^{\bullet}_{\ast}(K^{\gp},-)$.
	\item
	For a finite extension $E/k_{\gp}$, we define
	 $\tH^{\bullet}_{\ast}(K^{\gp},E)_{\la}$ to be the subspace
	 of $k_{\gp}$-locally analytic vectors in $\tH^{\bullet}(K^{\gp},E)$
	 (see \cite{emertonAnalytic}).
	The Lie algebra $\gg$ of $\bG$
	acts on the subspace $\tH^{\bullet}_{\ast}(K^{\gp},E)_{\la}$.
\end{enumerate}
Emerton proved (Theorem 2.2.11 of \cite{emerton})
 that his spaces are related to the classical cohomology groups
 by the following spectral sequence of smooth $\bG(k_{\gp})\times\cH(K^{\gp},E)$
 representations:
\begin{equation}
	\label{specseq}
	E_{2}^{p,q}
	=
	\Ext^{p}_{\gg}(\check W, \tH^{q}_{\ast}(K^{\gp},E)_{\la})
	\implies
	H^{p+q}_{\ast,\cl}(K^{\gp},W),
\end{equation}
where $\check W$ is the contragredient representation.

The spectral sequence shows that
 the spaces $\tH^{\bullet}_{\ast}(K^{\gp},-)$
 carry the information of all the systems of Hecke eigenvalues of
 automorphic representations of $\bG$ of cohomological type with tame level $K^{\gp}$.
On the other hand, the spaces $\tH^{\bullet}_{\ast}(K^{\gp},E)$ are not too big.
More precisely, they are \emph{continuous admissible}
 representations of $\bG(k_{\gp})$.
This means that the continuous dual space is finitely generated
 as a module over the Iwasawa algebra of a
 compact open subgroup of $\bG(k_{\gp})$
(for this definition, see \cite{schneider-teitelbaum-2002}
 or Definition 7.2.1 of \cite{emertonAnalytic}.)
As a consequence, Emerton was able in some cases to use his spaces
 to interpolate the systems of Hecke-eigenvalues into an ``eigenvariety''
 (see \cite{emerton} and also \cite{hill-eigenvarieties}
 for a number of examples when Emerton's method is successful).
In a rather different vein, the spaces $\tH^{\bullet}_{\ast}(K^{\gp},E)$
 have been used in \cite{calegari-emerton-2007pre}
 to give strong bounds on the multiplicities of automorphic representations of
 cohomological type.
The latter results are reminiscent of Iwasawa's bounds
 on the growth of class numbers.

It is easy to summarise what is known about the spaces $\tH^{\bullet}_{\ast}$.
If we let $d$ be the dimension of the arithmetic quotients $Y(K_{f})$,
then clearly $\tH^{n}_{\ast}(\bG,-)$ is zero for $n>d$.
Furthermore, Emerton showed that $\tH^{d}_{\cpct}(\bG,-)=\tH^{d}(\bG,-)$.
Emerton calculated \cite{emerton} the groups $\tH^{\bullet}_{\ast}(\bG,-)$ when $\bG$
is either $\SL_{2}/\Q$ or $\GL_{2}/\Q$ or a real rank zero group.
The case where $\bG$ is simple and simply connected
is studied in the preprint \cite{hill-eigenvarieties}.
For such groups one knows that $\tH^{0}(\bG,\Q_{p})=\Q_{p}$ and
$$
	\tH^{1}(\bG,\Q_{p})
	=
	\Hom_{\cts}(\Cong(\bG),\Q_{p})_{\bG(\A_{f}^{\gp})-\smooth},
$$
where $\Cong(\bG)$ denotes the congruence kernel of $\bG$.
Beyond this, nothing is written on the subject.
In this paper we shall give a number of new results
on Emerton's spaces.

\subsection{Results and organization of the paper}

Our first results are expressions for Emerton's $p$-adic cohomology
 groups in terms of more familiar objects.
We show in \S\ref{interpret1sec}
 Theorem \ref{reinterpret1} and Corollary \ref{reinterpretusual}
 that $\tH^{\bullet}(K^{\gp},\Z_{p})$ is isomorphic to the sheaf cohomology
 $H^{\bullet}(Y(K^{\gp}),\scC)$,
 where $Y(K^{\gp})$ is the topological space
 $\bG(k)\backslash\bG(\A)/K_{\infty}^{\circ}K^{\gp}$,
 and $\scC$ is the sheaf of continuous $\Z_{p}$-valued functions.
Such functions need not be locally constant,
 since $Y(K^{\gp})$ is a projective limit of arithmetic quotients.
We also identify $\tH^{n}_{\cpct}(K^{\gp},\Z_{p})$ as a relative cohomology group.
These results allow us to use the machinery of homological
algebra to study the spaces $\tH^{n}_{\ast}$, since they are identified as
derived functors. In particular, we show (Corollary \ref{longexactsequence})
that there is a long exact sequence:
\begin{equation}
	\label{exactseq1}
	\to
	\tH^{n}_{\cpct}(K^{\gp},\Z_{p})
	\to
	\tH^{n}(K^{\gp},\Z_{p})
	\to
	\tH^{n}_{\partial}(K^{\gp},\Z_{p})
	\to
	\tH^{n+1}_{\cpct}(K^{\gp},\Z_{p})
	\to.
\end{equation}
Here we are using the notation
$$
	\tH^{n}_{\partial}(K^{\gp},\Z_{p})
	=
	\limprojs
	\limdKp
	H^{n}(\partial Y(K_{\gp}K^{\gp})^{\BS},\Z/p^{s}),
$$
where $\partial Y(K_{f})^{\BS}$ is the ``Borel--Serre'' boundary of $Y(K_{f})$.
More precisely,
 recall that the arithmetic quotients $Y(K_{f})$ are not necessarily compact.
However we may embed $Y(K_{f})$ as a dense open subspace of a finite
simplicial complex $Y(K_{f})^{\BS}$, such that the inclusion is
a homotopy equivalence.
We are defining $\partial Y(K_{f})^{\BS}=Y(K_{f})^{\BS}\setminus Y(K_{f})$.
In the case that $\bG$ is reductive and the centre of $\bG$ has rank zero over $k$,
the space $Y(K_{f})^{\BS}$ is the Borel--Serre compactification of $Y(K_{f})$
(see \cite{borel-serre} or \cite{Borel-Ji-book}).

In \S\ref{interpret2sec} (Theorem \ref{reinterpret2} and its Corollaries),
 we identify Emerton's space $\tH^{\bullet}(K^{\gp},\Z_{p})$
 with $H^{\bullet}(Y(K_{f}),\scC(K_{\gp}))$,
 where $\scC(K_{\gp})$ is the local system on $Y(K_{f})$ consisting of
 continuous functions $K_{\gp}\to\Z_{p}$.
The action of the fundamental group of a connected component of $Y(K_{f})$
on $\scC(K_{\gp})$ is by left translation on $K_{\gp}$.
The action of $K_{\gp}$ on these cohomology groups is given by
right translation on $\scC(K_{\gp})$.
This second interpretation has two technical advantages.
First, the space $Y(K_{f})$ is homotopic to a finite simplicial complex
 rather than a projective limit of finite simplicial complexes.
Second, the sheaf $\scC(G)$ is a local system, whereas $\scC$ is not.
In principle, this second interpretation allows us, given a triangulation
of $Y(K_{f})$, to calculate the spaces $\tH^{\bullet}_{\ast}$ as $K_{\gp}$-modules.

As an application, we give in \S\ref{spectralseqsec}
 a new proof of Emerton's spectral sequence (\ref{specseq}),
 starting from the definition
 $\tH^{\bullet}(K^{\gp},\Z_{p})=H^{\bullet}(Y(K_{f}),\scC(K_{\gp}))$.
We show that this is an example of a Grothendieck
 spectral sequence for (roughly speaking)
 the following composition of left-exact functors:
$$
	\begin{array}{c}
		\hbox{sheafs of continuous}\\
		\hbox{$K_{\gp}$-modules on $Y(K_{f})$}
	\end{array}
	\;\longrightarrow\;
	\begin{array}{c}
		\hbox{continuous}\\
		\hbox{$K_{\gp}$-modules}
	\end{array}
	\;\longrightarrow\;
	\begin{array}{c}
		\hbox{smooth}\\
		\hbox{$K_{\gp}$-modules.}
	\end{array}
$$
The first functor here is ``global sections'' and the second is $\Hom_{\gg}(W,-_{\la})$.
In fact our result is more general than Emerton's in the compactly supported case
(see Theorem 2.1.12 of \cite{emerton}).
Recall that Emerton required his spaces to be topological manifolds.
This is because he uses Poincar\'e duality to replace compactly supported
cohomology by homology in his proof.
We do not make this restriction.
The greater generality of our result (Theorem \ref{myspecseq})
 makes these techniques applicable (for example)
 to $S$-arithmetic quotients, or to subsets of the Borel--Serre boundary.
We also obtain a spectral sequence in the case that the finite dimensional
space $W$ is replaced by any continuous admissible representation.

In \S\ref{abstractsection} we calculate the spaces $\tH^{\bullet}_{\ast}$
 in some easy cases.
We show (Theorem \ref{unipotent})
that for a unipotent group $\bN$ over $\Q$, we have $\tH^{0}_{\ast}(\bN,\Z_{p})=\Z_{p}$
and $\tH^{n}_{\ast}(\bN,\Z_{p})=0$ for $n>0$.
We also show (Theorem \ref{torus} and Corollary \ref{leopoldt})
that a corresponding result for tori over $\Q$
is equivalent to Leopoldt's conjecture.
As a consequence, we show (Theorem \ref{parabolic})
 that if $\bN$ is the unipotent radical of a group $\bG/\Q$,
and $\bH=\bG/\bN$, then we may identify $\tH^{\bullet}_{\ast}(\bG,\Z_{p})$
with $\tH^{\bullet}_{\ast}(\bH,\Z_{p})$.

In \S\ref{borelserresection},
 we apply these results to the cohomology $\tH^{\bullet}_{\partial}$
 of the Borel--Serre boundary.
To describe these results, we need to introduce some notation.
Let $\bG$ be a reductive group over $\Q$, and assume that the maximal split torus
in the centre of $\bG$ is trivial.
Given any parabolic subgroup $\bP\subset \bG$ defined over $\Q$,
there is Levi decomposition $\bP=\bL\ltimes \bN$.
Here $\bN$ is the unipotent radical of $\bP$ and
$\bL$ is a Levi factor.
We further decompose $\bL=\bA\bM$,
 where $\bA$ is the maximal split torus in the centre of $\bL$
 and $\bM$ is the intersection of the kernels of the
 homomorphisms $\bL\to\GL_{1}/\Q$.
We define constants $\cd(\bG,\Z_{p})$, $\cd(\bG,\Q_{p})$,
 $\cd_{\cpct}(\bG,\Z_{p})$, $\cd_{\cpct}(\bG,\Q_{p})$,
 $\cd_{\partial}(\bG,\Z_{p})$, $\cd_{\partial}(\bG,\Q_{p})$,
 $D(\bG,\Z_{p})$ and  $D(\bG,\Q_{p})$ as follows:
\begin{eqnarray*}
	\cd_{\ast}(\bG,-)
	&=&
	\max\{ n :
	\tH^{n}_{\ast}(\bG,-) \ne 0
	\},\\
	D(\bG,-)
	&=&
	\max
	\left\{ \cd_{\cpct}(\bM,-)\; :
	\begin{array}{c}
		\hbox{$\bP=\bM\bA\bN \subset \bG$ is a}\\
		\hbox{proper parabolic subgroup}
	\end{array}
	\right\}.
\end{eqnarray*}
(The notation ``$\cd$'' stands for ``cohomological dimension''.)
We trivially have $\cd_{\ast}(\bG,\Q_{p})\le \cd_{\ast}(\bG,\Z_{p})$.
In fact these numbers are equal in all examples calculated by the author.
We shall prove (Theorem \ref{borelserreboundary}) that
$$
	\cd_{\partial}(\bG,-) \le D(\bG,-),
$$
where $-$ denotes either $\Z_{p}$ or $\Q_{p}$.
It seems rather likely that this bound is sharp; this will be clear from the proof.
As a consequence of the exact sequence (\ref{exactseq1})
we deduce that the map $\tH^{n}_{\cpct}(\bG,\Z_{p})\to \tH^{n}(\bG,\Z_{p})$ is
surjective for $n=D(\bG)+1$ and an isomorphism for $n\ge D(\bG)+2$.
This bound is rather strong: note that for groups $\bG/\Q$ of real rank zero,
 $D(\bG)=0$.
This is perhaps surprising, since the classical cohomology groups
$H^{n}_{\cl}$ and $H^{n}_{\cpct,\cl}$ may differ up to the dimension
of $Y(K_{f})$.

It is clear that $\cd(\bG,\Z_{p})$ is at most the virtual cohomological dimension
of arithmetic subgroups of $\bG$.
It is proved in \cite{borel-serre},
 that this virtual cohomological dimension is $\dim_{\R}(Y(K_{f}))-\rank_{\Q}(\bG)$.
We therefore have $\cd(\bG,\Z_{p})\le\dim_{\R}(Y(K_{f}))-\rank_{\Q}(\bG)$.
Using this, together with our previous result,
 we prove (Theorem \ref{vanishingtheorem})
 that
$$
	\cd_{\cpct}(\bG,\Z_{p})\le \dim_{\R}(Y(K_{f}))-\rank_{\Q}(\bG).
$$
This bound is certainly not sharp in general,
 although it may be sharp when $\bG$ is split.
For example, if $\bG$ has positive real rank, then
 we show (Proposition \ref{topdimension})
 that $\cd_{\cpct}(\bG,\Z_{p}),\cd(\bG,\Z_{p})\le \dim(Y(K_{f}))-1$,
 regardless of the rational rank.
If, in addition, the congruence kernel of $\bG$ is finite,
 then we show (Proposition \ref{nextdimensiondown})
 that $\cd(\bG,\Z_{p}),\cd_{\cpct}(\bG,\Z_{p})\le \dim(Y(K_{f}))-2$.
In this context, we recall that Serre conjectured (for $\bG$ simple and simply connected)
 that the congruence kernel is finite if and only if the real rank is at least $2$.
Based on this evidence, one might expect that in fact
 $\cd_{\cpct}(\bG,\Z_{p}),\cd(\bG,\Z_{p})\le \dim(Y(K_{f}))-\rank_{\R}(\bG)$.

Finally, we show (Proposition \ref{vanishingtheoremconj})
 that $\cd(\bG,-)\le \dim(Y(K_{f}))-\rank_{\R}(\bG)$ for all
reductive groups $\bG$ whose centre has rational rank $0$, if and only if
 the same bound also holds for $\cd_{\cpct}(\bG,-)$.
Again, this follows from our result on the cohomology of the boundary,
together with the long exact sequence (\ref{exactseq1}).

\section{Relation to Sheaf Cohomology}
\label{interpret1sec}

\subsection{Some facts about \v Cech cohomology}

Let $X$ be a topological space and $\cF$ a presheaf on $X$.
For an open cover $\gU=\{U_{i}: i\in I\}$ of $X$,
 one defines the \v Cech complex $\check C^{\bullet}(\gU,\cF)$ by
$$
	\check C^{n}(\gU,\cF)
	=
	\{(f_{i_{0},\ldots,i_{n}})_{i_{0},\ldots,i_{n}\in I^{n+1}} : 
	f_{i_{0},\ldots,i_{n}}\in \cF(U_{i_{0}}\cap \cdots \cap U_{i_{n}})\}.
$$
The cohomology groups of this complex are written $\check H^{\bullet}(\gU,\cF)$.
The \v Cech cohomology groups are defined to be
the direct limits of these cohomology groups:
$$
	\check H^{n}(X,\cF)
	=
	\lim_{\stackrel{\textstyle \to}{\scriptstyle \gU}} \check H^{n}(\gU,\cF).
$$
In fact $\check H^{n}(X,\cF)$ depends only on the sheafification of $\cF$.
An open cover $\gU$ of $X$ is called \emph{$\cF$-acyclic}
 (or sometimes a \emph{Leray cover}) if, for any intersection $U$
 of finitely many elements of $\gU$, we have $\check H^{n}(U,\cF)=0$ for all $n>0$.

\begin{theorem}[Leray's Theorem]
	If the cover $\gU$ is $\cF$-acyclic, then
	$\check H^{n}(X,\cF)=\check H^{n}(\gU,\cF)$.
\end{theorem}

\begin{theorem}[Thm. III.4.12 of \cite{bredon}]
	If $\cF$ is a sheaf on a paracompact Hausdorff
	topological space $X$, then the \v Cech cohomology groups
	of $\cF$ are equal to its sheaf cohomology groups,
	i.e. the derived functors of the global sections functor.
\end{theorem}

Given a presheaf $\cF$ on a topological space $Y$,
and a subspace $Z\subset Y$, we define presheafs $\cF_{Z}$ and $\cF^{Z}$
 on $X$ by
$$
	\cF_{Z}(U)
	=
	\begin{cases}
		\cF(U) & \hbox{if } U\cap Z\ne \emptyset\\
		0 & \hbox{otherwise,}
	\end{cases}
	\qquad\quad
	\cF^{Z}(U)
	=
	\begin{cases}
		0 & \hbox{if } U\cap Z\ne \emptyset\\
		\cF(U) & \hbox{otherwise.}
	\end{cases}
$$
It turns out that $\check H^{\bullet}(Z,\cF)=\check H^{\bullet}(X,\cF_{Z})$,
and one defines $\check H^{\bullet}(Y,Z,\cF)=\check H^{\bullet}(Y,\cF^{Z})$.
There is a short exact sequence of presheafs:
$$
	0
	\to
	\cF^{A}
	\to
	\cF
	\to
	\cF_{A}
	\to
	0.
$$
This gives a long exact sequence:
$$
	\check H^{n}(\gU,\cF^{A}) \to
	\check H^{n}(\gU,\cF) \to
	\check H^{n}(\gU,\cF_{A}) \to
	\check H^{n+1}(\gU,\cF^{A}).
$$
Passing to the direct limit, we obtain
 the long exact sequence of \v Cech cohomology groups:
$$
	\check H^{n}(X,A,\cF) \to
	\check H^{n}(X,\cF) \to
	\check H^{n}(A,\cF) \to
	\check H^{n+1}(X,A,\cF).
$$
If $A$ is an abelian group, then we shall write
$\underline A$ for the sheaf of locally constant $A$-valued functions.
Using Leray's theorem, one easily proves the following:

\begin{theorem}[Comparison Theorem]
	\label{comparisontheorem}
	Let $Y$ be a finite simplicial complex,  and $Z\subset Y$
	a subcomplex.
	For any abelian group $A$,
	 we have $\check H^{\bullet}(Y,Z, \underline A)=H^{\bullet}(Y,Z,A)$,
	where the right hand side is singular cohomology.
\end{theorem}

In fact the comparison theorem holds for much more general topological
spaces (see for example \cite{spanier}).

\begin{theorem}[Lem. 6.6.11, Cor 6.1.11 and Cor. 6.9.9 of \cite{spanier}]
	\label{spantheorem}
	Let $Y$ be a finite simplicial complex and
	$Z$ a subcomplex.
	For any abelian group $A$,
	 we have $H^{\bullet}_{\cpct}(Y\setminus Z,A)=H^{\bullet}(Y,Z,A)$.
\end{theorem}

\subsection{Emerton's Formalism}
\label{emerton-formal}
Emerton used the following formalism to introduce the
groups $\tH^{n}_{\ast}$.
Let $G$ be a compact, $\Q_{p}$-analytic group,
and fix a basis of open, normal subgroups:
$$
	G
	=
	G_{0}
	\supset
	G_{1}
	\supset
	\ldots.
$$
Suppose we have a sequence of simplicial maps between
finite simplicial complexes
$$
	\cdots \to Y_{2}
	\to Y_{1}
	\to Y_{0},
$$
and subcomplexes:
$$
	\cdots \to Z_{2} \to Z_{1} \to Z_{0},
$$
each equipped with a right action of $G$, and satisfying the following
conditions:
\begin{enumerate}
	\item
	the maps in the sequence are $G$-equivariant;
	\item
	$G_{r}$ acts trivially on $Y_{r}$;
	\item
	if $0\le r' \le r$ then the maps $Y_{r}\to Y_{r'}$ and $Z_{r}\to Z_{r'}$
	are Galois covering maps with deck transformations
	provided by the natural action of $G_{r'}/G_{r}$ on $Y_{r}$.
\end{enumerate}
Given this data, we let $Y$ be the projective limit of the spaces $Y_{r}$,
 and $Z$ be the projective limit of the spaces $Z_{r}$.
We shall use the notation $Y^{\circ}=Y\setminus Z$,
 $Y_{i}^{\circ}=Y_{i}\setminus Z_{i}$.
Emerton defined the following spaces:
$$
	\tH^{n}_{\cpct}(Y^{\circ},\Z_{p})
	=
	\limprojs
	\limdr
	H^{n}_{\cpct}(Y_{r}^{\circ},\Zps),
	\quad
	\tH^{n}_{\cpct}(Y^{\circ},\Q_{p})
	=
	\tH^{n}_{\cpct}(Y^{\circ},\Z_{p})
	\otimes_{\Z_{p}}
	\Q_{p}.
$$
$$
	\tH^{n}(Y,\Z_{p})
	=
	\limprojs
	\limdr
	H^{n}(Y_{r},\Zps),
	\quad
	\tH^{n}(Y,\Q_{p})
	=
	\tH^{n}(Y,\Z_{p})
	\otimes_{\Z_{p}}
	\Q_{p}.
$$
These spaces are $G$-modules, where $G$ acts via its action on the spaces $Y_{r}$.
In applications, $G$ will be a compact open subgroup of $\bG(k_{\gp})$,
and the space $Y$ will be either $Y(K^{\gp})^{\BS}$
or $\partial Y^{\BS}$. If $Y=Y(K^{\gp})^{\BS}$ then we sometimes take
$Z=\partial Y(K^{\gp})^{\BS}$, and sometimes $Z$ is the empty set.
We shall write $\scC$ for the sheaf on $Y$ of continuous $\Z_{p}$-valued
functions.

\begin{theorem}
	\label{reinterpret1}
	With the above notation,
	there is a canonical isomorphism of $G$-modules:
	$$
		\tH^{n}_{\cpct}(Y^{\circ},\Z_{p})=\check H^{n}(Y, Z, \scC).
	$$
\end{theorem}

\begin{proof}
To prove the theorem,
 we construct acyclic covers of $Y$ and $Y_{r}$ and apply Leray's Theorem.
The proof will be broken up into manageable pieces.

\subsection{A cover}

We first choose a finite open cover $\gU$
 of $Y_{0}$ with the following properties:
\begin{enumerate}
	\item
	If $U$ is an intersection of finitely many sets in $\gU$
	then either $U$ is empty or $U$ is contractible.
	\item
	If $U$ is an intersection of finitely many sets in $\gU$
	and $U\cap Z_{0}$ is non-empty,
	then $U\cap Z_{0}$ is a deformation retract of $U$.
\end{enumerate}
For each $U\in \gU$, we let $U^{(r)}$ be the preimage
of $U$ in $Y_{r}$.
The sets $U^{(r)}$ form an open cover $\gU^{(r)}$ of $Y_{r}$,
and have the following properties:
\begin{enumerate}
	\item
	For every $U_{1}^{(r)},\ldots,U_{s}^{(r)}\in \gU^{(r)}$
	with non-empty intersection, the intersection
	$U_{1}^{(r)}\cap \cdots \cap U_{s}^{(r)}$
	is isomorphic as a topological $G$-set to
	$(U_{1}\cap \cdots \cap U_{r})\times (G/G_{r})$.
	In particular, the intersection is homotopic to a
	finite set.
	\item
	If $U_{1}^{(r)},\ldots,U_{s}^{(r)}\in \gU^{(r)}$
	and $U_{1}^{(r)}\cap \cdots \cap U_{s}^{(r)}\cap Z_{r}$ is
	non-empty,
	then	$U_{1}^{(r)}\cap \cdots \cap U_{s}^{(r)}\cap Z_{r}$
	is a deformation retract of $U_{1}^{(r)}\cap \cdots \cap U_{s}^{(r)}$.
\end{enumerate}
Furthermore, for each set $U\in \gU$, we define $\tilde U$
to be the preimage of $U$ in $Y$.
The sets $\tilde U$ form an open cover $\tilde\gU$ of $Y$.
We immediately verify the following:
\begin{enumerate}
	\item
	if $\tilde U_{1},\ldots,\tilde U_{s}\in\tilde \gU$ have non-empty intersection,
	then their intersection is equivalent as a topological $G$-set
	to $(U_{1}\cap \cdots\cap U_{s})\times G$.
	\item
	if $\tilde U_{1},\ldots,\tilde U_{s}\in\tilde \gU$ and
	$\tilde U_{1} \cap\cdots \cap\tilde U_{s}\cap Z$ is non-empty,
	then $\tilde U_{1} \cap\cdots \cap\tilde U_{s}\cap Z$ is a deformation retract
	of $\tilde U_{1} \cap\cdots \cap\tilde U_{s}$.
\end{enumerate}

\subsection{$\gU^{(r)}$ is $(\uZps)^{Z_{r}}$-acyclic}

Let $U$ be an intersection of finitely many sets in $\gU$,
and let $U^{(r)}$ be the preimage of $U$ in $Y_{r}$.
We know that $U$ is contractible, and $U^{(r)}=U\times (G/G_{r})$.
The sheaf $(\uZps)^{Z_{r}}$ on $Y_{r}$ consists of locally constant
$\Zps$-valued functions, which vanish on $Z_{r}$.
It follows that $\check H^{\bullet}(U^{(r)},(\uZps)^{Z_{r}})$ is a direct sum of finitely
many copies of $\check H^{\bullet}(U,(\uZps)^{Z_{0}})$.
We must therefore show that $\check H^{n}(U,(\uZps)^{Z_{0}})=0$ for all $n>0$.

If $U\cap Z_{0}$ is empty, then we have
$\check H^{n}(U,(\uZps)^{Z_{0}})=\check H^{n}(U,\uZps)$.
By the comparison theorem, this is the same as singular cohomology, and therefore
only depends on $U$ up to homotopy.
Since $U$ is contractible, it follows that $\check H^{n}(U,\uZps)=0$ for $n>0$.

Suppose instead that $U\cap Z_{0}$ is non-empty.
In this case, we know that $U\cap Z_{0}$ is a deformation retract of $U$.
It follows that the restriction map
$H^{\bullet}_{\sing}(U,\Zps) \to H^{\bullet}_{\sing}(U\cap Z_{0},\Zps)$ is an isomorphism.
By the comparison theorem, it follows that the map
$\check H^{\bullet}(U,\uZps)
 \to \check H^{\bullet}(U\cap Z_{0},\uZps)$ is an isomorphism.
The long exact sequence shows that $\check H^{\bullet}(U,U\cap Z_{0},\uZps)=0$.

\subsection{$\tilde \gU$ is $\scC$-acyclic}

Let $U$ be an intersection of finitely many sets in $\gU$,
and let $\tilde U$ be the preimage of $U$ in $Y$.
We know that $U$ is contractible, and $\tilde U=U\times G$.
We must show that $\check H^{n}(\tilde U,\scC)=0$ for $n>0$.

Let $\tilde \gV$ be an open cover of $\tilde U$, and choose an element
 $\sigma\in \check H^{n}(\tilde \gV,\scC)$ with $n>0$.
We shall find a refinement $\tilde \gW$ of $\tilde \gV$, such that
 the image of $\sigma$ in $\check H^{\bullet}(\tilde \gW,\scC)$ is zero.
By passing to a refinement of $\tilde \gV$ if necessary, we may assume that
$\tilde\gV$ is finite, and that each element of $\tilde \gV$
 is of the form $V_{i}\times H_{i}$
 for some open subset $V_{i}\subset U$ and some open coset $H_{i}\subset G$.
By refining still further, we may assume that the cosets $H_{i}$ are all cosets of
the same open subgroup $G_{r}\subset G$.
This means that $\tilde \gV$ is the pullback
 of an open cover $\gV^{(r)}$ of $U^{(r)}$.
For an open subset $V^{(r)}\subset U^{(r)}$, we have
$$
	\scC(\tilde V^{(r)})
	=
	S(V^{(r)}),
$$
where $S$ is the constant sheaf on $U^{(r)}$
with values in $\scC(G_{r})$
and $\tilde V^{(r)}$ is the preimage of $V^{(r)}$ in $\tilde U$.
It follows that
$$
	\check H^{\bullet}(\tilde\gV, \scC)
	=
	\check H^{\bullet}(\gV^{(r)}, S).
$$
Since $U^{(r)}$ is homotopic to a finite set and $S$ is a constant sheaf,
 it follows that $\check H^{>0}(U^{(r)}, S)$ is zero.
This implies there is a refinement $\gW^{(r)}$ of $\gV^{(r)}$, such that
 the image of $\sigma$ in $\check H^{\bullet}(\gW^{(r)}, S)$ is zero.
Pulling $\gW^{(r)}$ back to $\tilde U$,
 we have a refinement $\tilde \gW$ of $\tilde\gV$,
 such that the image of $\sigma$ in $\check H^{\bullet}(\tilde\gW,\scC)$
 is zero.

\subsection{$\tilde \gU$ is $\scC^{Z}$-acyclic}

Let $U$ be an intersection of finitely many sets in $\gU$,
and let $\tilde U$ be the preimage of $U$ in $Y$.
We know that $U$ is contractible, and $\tilde U=U\times G$.
We must show that $\check H^{n}(\tilde U,\tilde U\cap Z,\scC)=0$ for $n>0$.
If $U$ does not intersect $Z_{0}$,
 then this follows from the previous part of the proof.
We therefore assume that $U$ intersects $Z_{0}$.
In this case, we know that $U\cap Z_{0}$ is a deformation retract of $U$.
In particular, $U\cap Z_{0}$ is contractible, and
$\tilde U\cap Z= (U\cap Z_{0})\times G$.
The previous part of the proof shows that $\check H^{>0}(\tilde U, \scC)=0$
and $\check H^{>0}(\tilde U \cap Z,\scC)=0$.
Furthermore, one sees immediately that the restriction map
 $\check H^{0}(\tilde U,\scC)\to \check H^{0}(\tilde U \cap Z,\scC)$
 is an isomorphism.
Hence by the long exact sequence,
 we have $H^{\bullet}(\tilde U,\tilde U \cap Z,\scC)=0$.

\subsection{}
Fix for a moment a cohomological degree $n$,
and let $U_{1},\ldots,U_{N}$ be the non-empty intersections
of $n+1$-tuples of sets in $\gU$, for which $U_{i}\cap Z_{0}=\emptyset$.
For each $U_{i}$, we let $U_{i}^{(r)}$ be the preimage of $U_{i}$ in $Y_{r}$
and $\tilde U_{i}$ be the preimage of $U_{i}$ in $Y$.

Recall that $\check H^{\bullet}(\gU^{(r)},(\Z/p^{r})^{Z})$
 is the cohomology of the chain complex
$$
	\check C^{n}(\gU^{(r)},(\uZps)^{Z_{r}})
	=
	\prod_{i=1}^{N} (\uZps)^{Z_{r}}(U_{i}^{(r)}).
$$
Each $U_{i}$ is contractible and disjoint from $Z_{0}$.
Furthermore $U_{i}^{(r)}=U_{i}\times (G/G_{r})$,
 so we have an isomorphism of $G$-modules:
$(\uZps)^{Z_{r}}(U_{i}^{(r)})=(\uZps)(G/G_{r})$.
This gives
$$
	\check C^{n}(\gU^{(r)},(\uZps)^{Z_{r}})
	=
	\left((\uZps)(G/G_{r})\right)^{N}.
$$
Similarly, we have 
$$
	\check C^{n}(\tilde \gU,(\uZps)^{Z})
	=
	\left((\uZps)(G)\right)^{N}.
$$
Comparing the two formulae, it is clear that
$$
	\check C^{\bullet}(\tilde \gU,(\uZps)^{Z})
	=
	\limdr \check C^{\bullet}(\gU^{(r)},(\uZps)^{Z_{r}}).
$$
Since the functor $\limdr$ is exact, we have
$$
	\check H^{\bullet}(\tilde \gU,(\uZps)^{Z})
	=
	\limdr \check H^{\bullet}(\gU^{(r)},(\uZps)^{Z_{r}}).
$$

\subsection{}
Note also that $\scC^{Z}(\tilde U_{i})=\scC(G)$,
and so we have
$$
	\check C^{n}(\tilde \gU,\scC^{Z})
	=
	\scC(G)^{N}.
$$
It follows that $\check C^{n}(\tilde \gU,\scC^{Z})$
 is an admissible $\Z_{p}[G]$-module in the sense of
 Definition 1.2.1 of \cite{emerton}.
Furthermore we have:
$$
	\check C^{\bullet}(\tilde \gU,\scC^{Z})
	=
	\limprojs \check C^{\bullet}(\tilde \gU,(\uZps)^{Z}),
	\quad
	\check C^{\bullet}(\tilde \gU,(\uZps)^{Z})
	=
	\check C^{\bullet}(\tilde \gU,\scC^{Z})/p^{s}.
$$
Hence by Proposition 1.2.12 of \cite{emerton}, we have:
$$
	\check H^{\bullet}(\tilde \gU,\scC^{Z})
	=
	\limprojs \check H^{\bullet}( \tilde \gU,(\uZps)^{Z}).
$$
By the previous part of the proof, we have:
$$
	\check H^{\bullet}(\tilde \gU,\scC^{Z})
	=
	\limprojs \limdr
	\check H^{\bullet}( \gU^{(r)},(\uZps)^{Z_{r}}).
$$
Since our covers are acyclic, this translates to
$$
	\check H^{\bullet}(Y,Z,\scC)
	=
	\limprojs \limdr
	\check H^{\bullet}( Y_{r},Z_{r},\uZps).
$$
On the other hand, by Theorem \ref{spantheorem},
we have
$$
	\check H^{\bullet}( Y_{r},Z_{r},\uZps)
	=
	H^{\bullet}_{\cpct}( Y_{r}\setminus Z_{r},\Zps).
$$
The result follows.
\end{proof}

\begin{corollary}
	\label{reinterpretusual}
	With the above notation,
	$\tH^{n}(Y,\Z_{p})=\check H^{n}(Y, \scC)$.
\end{corollary}

\begin{proof}
We apply the theorem in the case that $Z$ is empty.
Since $Y^{\circ}=Y$, which is compact,
 it follows that usual cohomology is the same as compactly supported cohomology on
 each $Y_{r}$.
\end{proof}

\begin{corollary}
	\label{longexactsequence}
	In the notation of the introduction, there is a long exact sequence:
	$$
		\tH^{n}_{\cpct}(K^{\gp},\Z_{p})
		\to
		\tH^{n}(K^{\gp},\Z_{p})
		\to
		\tH^{n}_{\partial}(K^{\gp},\Z_{p})
		\to
		\tH^{n+1}_{\cpct}(K^{\gp},\Z_{p}).
	$$
\end{corollary}

\begin{proof}
We have shown above that
\begin{eqnarray*}
	\tH^{\bullet}(K^{\gp},\Z_{p})
	&=&
	\check H^{\bullet}(Y(K^{\gp}),\scC),\\
	\tH^{\bullet}_{\partial}(K^{\gp},\Z_{p})
	&=&
	\check H^{\bullet}(\partial Y(K^{\gp})^{\BS},\scC),\\
	\tH^{\bullet}_{\cpct}(K^{\gp},\Z_{p})
	&=&
	\check H^{\bullet}(Y(K^{\gp})^{\BS},\partial Y(K^{\gp})^{\BS},\scC).
\end{eqnarray*}
The result follows from the long exact sequence in \v Cech cohomology.
\end{proof}

\subsection{A local system}
\label{interpret2sec}

We keep the notation introduced in \S\ref{emerton-formal}.
We have a finite simplicial complex $Y_{0}$ and a
profinite simplicial complex $Y$, together with a map
$Y\to Y_{0}$. This map is a $G$-bundle.
We shall assume for a moment that $Y_{0}$
 is connected and write $\Gamma$ for the fundamental group of $Y_{0}$.
We regard $\Gamma$ as acting on the left on the universal cover $Y^{\univ}$ of $Y$.
As was noted by Emerton (\S2.1 of \cite{emerton}), this data corresponds to giving
a group homomorphism $\Gamma\to G$.
The space $Y$ may then be constructed as follows:
$$
	Y
	=
	\Gamma\backslash(Y^{\univ}\times G),
$$
where the action of $\Gamma$ on $Y^{\univ}\times G$
 is $\gamma (y,g)=(\gamma y,\gamma g)$.
The right action of $G$ on $Y$ is $(y,g)h=(y,gh)$.

Let $\scC(G)$ be the space of continuous functions $G\to\Z_{p}$.
We have an action of $\Gamma$ on $\scC(G)$ by left translation.
By abuse of notation, we shall also write $\scC(G)$ for the corresponding
local system on $Y_{0}$.
We also have an action of $G$ on $\scC(G)$ by right translation.
This gives the local system $\scC(G)$ the structure of a sheaf of
 $G$-modules.

If we again allow $Y_{0}$ to have finitely many connected components,
then we may make the same construction of a local system $\scC(G)$
on each connected component.
Together these form a local system on the whole of $Y_{0}$.
With this notation we have:

\begin{theorem}
	\label{reinterpret2}
	There are canonical isomorphisms of $G$-modules:
	$$
		H^{\bullet}(Y,\scC)
		=
		H^{\bullet}(Y_{0},\scC(G)),
		\quad
		H^{\bullet}(Y,Z,\scC)
		=
		H^{\bullet}(Y_{0},Z_{0},\scC(G)).
	$$
\end{theorem}

\begin{remark}
	It is clear that $H^{\bullet}(Y,\scC)$
	 depends only on the profinite simplicial complex $Y$.
	The independence of $H^{\bullet}(Y_{0},\scC(G))$ of the
	level $Y_{0}$ follows from Shapiro's Lemma.
\end{remark}

\begin{remark}
	The formula of the theorem makes it in principle
	possible to calculate Emerton's groups $\tH^{\bullet}_{\ast}(K^{\gp},\Z_{p})$,
	at least as $K_{\gp}$-modules, as long as one has a triangulation
	of the space $Y(K^{\gp}K_{\gp})$, and as long as $K_{\gp}$ is torsion-free.
\end{remark}

\begin{proof}
We have a projection map $f:Y\to Y_{0}$.
There is a corresponding Leray spectral sequence:
$$
	H^{p}(Y_{0},R^{q}f_{*} \scC)
	\implies
	H^{p+q}(Y,\scC).
$$
In view of this, the theorem follows from the following two Lemmas.
\end{proof}

\begin{lemma}
	$f_{*}\scC=\scC(G)$ and $f_{*}(\scC^{Z})=\scC(G)^{Z}$.
\end{lemma}

\begin{proof}
For any contractible open subset $U\subset Y_{0}$, we have
$$
	f_{*}\scC(U)
	=
	\scC(\tilde U),
$$
where $\tilde U$ is the preimage of $U$ in $Y$.
Topologically we have $\tilde U=U\times G$.
Since $U$ is connected, we have an isomorphism
$$
	\scC(\tilde U)
	=
	\scC(G).
$$
\end{proof}

\begin{lemma}
	For $n>0$, $R^{n}f_{*}\scC=R^{n}f_{*}(\scC^{Z})=0$.
\end{lemma}

\begin{proof}
Recall that $R^{n}f_{*}\scC$ is the sheafification of the presheaf
$$
	U
	\mapsto
	H^{n}(\tilde U,\scC).
$$
If $U$ is contractible, then the arguments in \S2.5 and \S2.6
 shows that $H^{n}(\tilde U,\scC)=0$ for $n>0$.
Since any open cover of $Y_{0}$ has a refinement consisting of contractible sets,
the result follows.
\end{proof}

\begin{corollary}
	\label{groupcohomology}
	If $Y_{0}$ is a $K(\Gamma,1)$ space (i.e.
	 if $Y_{0}$ is connected and its universal cover is contractible)
	 then there is a canonical isomorphism of $G$-modules:
	$$
		H^{\bullet}(Y,\scC)
		=
		H^{\bullet}_{\Group}(\Gamma,\scC(G)).
	$$
\end{corollary}

\begin{corollary}
	In the notation of the introduction,
	let $K_{f}=K_{\gp}K^{\gp}$, where $K_{\gp}$ is torsion-free.
	Then we have as $K_{\gp}$-modules:
	\begin{eqnarray*}
		\tH^{\bullet}(K^{\gp},\Z_{p})
		&=&
		H^{\bullet}(Y(K_{f}),\scC(K_{\gp})),\\
		\tH^{\bullet}_{\cpct}(K^{\gp},\Z_{p})
		&=&
		H^{\bullet}_{\cpct}(Y(K_{f}),\scC(K_{\gp}))
		\;=\;
		H^{\bullet}(Y(K_{f})^{\BS},\partial Y(K_{f})^{\BS},
		\scC(K_{\gp})),\\
		\tH^{\bullet}_{\partial}(K^{\gp},\Z_{p})
		&=&
		H^{\bullet}(\partial Y(K_{f})^{\BS},
		\scC(K_{\gp})).
	\end{eqnarray*}
\end{corollary}

\section{Another proof of Emerton's spectral sequence}
\label{spectralseqsec}

We shall show in this section that Emerton's spectral sequence is an
example of a Grothendieck spectral sequence.
We are forced in this section to work over $\Q_{p}$ rather than $\Z_{p}$.
This is because $\scC(G)$ is not injective as a continuous
$\Z_{p}[G]$-module, whereas $\scC(G,\Q_{p})$ is injective in a suitable
category of representations of $G$.
The main result in this section is Theorem \ref{myspecseq}.
The results of this section are not required in the following sections.

\subsection{Relative homological algebra}
As before we shall write $G$ for a compact $p$-adic analytic group.
We write $\modG$ for the category of continuous representations of $G$ on
topological vector spaces over $\Q_{p}$.
Here ``continuous'' means that the map $G\times V \to V$
defining the action is continuous.
The morphisms in this category are defined to be the $G$-equivariant continuous linear maps.

The category $\modG$ is abelian, but does not have enough injectives; this
is because the category of topological vector spaces does not have enough injectives.
To get around this problem, one defines a \emph{strong morphism}
to be a morphism $f:V\to W$ in $\modG$, such that
the exact sequence
$$
	0 \to \ker f \to V \to W \to \coker f \to 0,
$$
is chain homotopic to the zero complex in the category of topological vector spaces.
This means that
\begin{enumerate}
	\item
	$\ker f$ and $\im f$ are closed topological direct summands;
	\item
	The continuous linear bijection $V/\ker f\to \im f$ is a homeomorphism.
\end{enumerate}
(For various subcategories, property (2) is immediate from (1)
by the open mapping theorem.)
An object $I$ of $\modG$ is called \emph{relatively injective} if,
given any strong injection $A\to B$ in $\modG$ and any morphism $f:A\to I$,
the function $f$ extends to a function $\tilde f:B\to I$.
By a \emph{relatively injective resolution of $V$}, we mean an exact sequence
of strong morphisms:
$$
	0
	\to
	V
	\to
	I^{1}
	\to
	I^{2}
	\to
	\cdots,
$$
where each $I^{n}$ is relatively injective.
Every object of $\modG$ has such a resolution,
 and any two such resolutions are chain homotopic.
In particular, if we write $\scC(G,V)$ for
the space of continuous functions from $G$ to $V$, equipped with the uniform topology,
then there is a canonical strong injection $V\to \scC(G,V)$
and $\scC(G,V)$ is relatively injective.

For a left exact functor $\bF:\modG\to \bC$,
one defines the relative (right-) derived functors $(R^{\bullet}_{\rel}\bF)(V)$ by
$$
	(R^{n}_{\rel}\bF)(V)
	=
	H^{n}(\bF(I^{\bullet})).
$$
This is independent of the resolution.
For example, one defines the continuous cohomology groups of $G$
 to be the relative derived functors of the functor $V\mapsto V^{G}$
 (see \cite{casselman-wigner} or \S IX.1.5 of \cite{borel-wallach} for this interpretation):
$$
	H^{n}_{\cts}(G,V)
	=
	(R^{n}_{\rel}(-^{G}))(V).
$$
This is the same as the cohomology of the usual complex of continuous
cochains $G\times \cdots \times G \to V$.
More generally for an object $W$ of $\modG$ we define:
$$
	\Ext_{G}^{n}(W,V)
	=
	(R^{n}_{\rel} \Hom_{\modG}(W,-)) (V).
$$
Short exact sequences of strong morphisms give rise to long exact sequences of
relative derived functors.

\subsection{Continuous admissible representations}
Recall that an \emph{admissible continuous $G$-module}
is a continuous representation of $G$ on a Banach space $V$ over $\Q_{p}$,
such that the continuous dual $V'$ of $V$ is finitely generated over the Iwasawa
algebra of $G$.
We write $\modaG$ for the full subcategory of continuous admissible representations
of $G$.
The duality functor taking $V$ to $V'$ is an antiequivalence of categories, between the
category of continuous admissible representations of $G$, and the category
of finitely generated modules over the Iwasawa algebra.
This duality takes $\scC(G,\Q_{p})$ to the rank $1$ free module.
Consequenctly, $\scC(G,\Q_{p})$ is injective in $\modaG$, and any injective object of
$\modaG$ is a direct summand of $\scC(G,\Q_{p})^{N}$ for some $N$.
From this we deduce the following:

\begin{lemma}
	\begin{enumerate}
		\item
		The category $\modaG$ has enough injectives;
		\item
		The injective objects of $\modaG$ are relatively injective in $\modG$.
	\end{enumerate}
\end{lemma}

The lemma implies that any injective resolution in $\modaG$
is a relatively injective resolution in $\modG$.
As a consequence of this we have the following:

\begin{lemma}
	Let $\bF:\modG\to \bC$ be a left exact functor.
	For any continuous admissible representation
	$V$ of $G$, we have $(R^{n}_{\rel}\bF)(V)=(R^{n}(\bF|_{\modaG}))(V)$.
\end{lemma}

In particular, the restrictions of $H^{n}_{\cts}(G,-)$ and $\Ext^{\bullet}_{G}(W,-)$
to $\modaG$ are derived functors, rather than relative derived functors.

\subsection{Smooth representations}
By a smooth representation of $G$,
we shall mean an abstract vector space $V$ over $\Q_{p}$,
equipped with an action of $G$ by linear maps,
such that every vector in $V$ has open stabilizer in $G$.
We shall write $\SmoG$ for the category of smooth
representations of $G$.
The morphisms are the $G$-equivariant linear maps.

Since every vector in a smooth representation has finite $G$-orbit,
it is clear that the irreducible smooth representations of $G$ are finite dimensional.
Every finite dimensional smooth representation is semi-simple.
Furthermore, every smooth representation of $G$ is an inductive limit
of finite dimensional smooth representations.
It follows that every smooth representation is a sum of finite dimensional irreducible
smooth representations.
Every short exact sequence in $\SmoG$ splits and every object is injective.

\subsection{}

For objects $W$ of $\modG$ and $X$ of $\SmoG$, we define an object
$X\otimes W$ of $\modG$.
As a vector space, this is the algebraic tensor product over $\Q_{p}$.
The $G$-action is on both $X$ and $W$.
To define the topology on $X\otimes W$, we
 identify $X\otimes W$ with the inductive limit of the spaces $F\otimes W$,
 where $F$ is a finite dimensional subspace of $X$.
By choosing a basis of $F$, we may identify $F\otimes W$ with  $W^{\dim F}$,
 and we equip $W^{\dim F}$ with the product topology.
The corresponding topology on $F\otimes W$ is independent of the choice of basis.
Equivalently, a subset $U\subset X\otimes W$ is a neighbourhood of $0$
if and only if for every $x\in X$ there is a neighbourhood $U_{x}$ of $0$ in $W$,
such that $x\otimes U_{x}\subset U$.

\begin{lemma}
	\label{lemmaexacttensor}
	The functor $-\otimes W:\SmoG\to\modG$
	 takes short exact sequences in $\SmoG$
	 to strong short exact sequences in $\modG$.
\end{lemma}

\begin{proof}
Suppose we have an exact sequence in $\SmoG$:
\begin{equation}
	\label{ses}
	0 \to A \to B \to C \to 0.
\end{equation}
Tensoring with $W$, we obviously still have an exact sequence of vector spaces
and $G$-equivariant linear maps.
The issue here is to show that $B\otimes W$ is the topological direct sum
of $A\otimes W$ and $C\otimes W$.
This follows because $B=A\oplus C$.
\end{proof}

\subsection{}
For objects $V,W$ of $\modG$, we use the notation
$$
	\Hom_{G-\st}(V,W)
	=
	\limdr\Hom_{G_{r}}(V,W).
$$
We regard this space as an abstract vector space over $\Q_{p}$;
it is an object of $\SmoG$ with the action of $G$ by conjugation.

\begin{lemma}
	\label{lemmasix}
	The functor $\Hom_{G-\st}(W,-):\modG\to \SmoG$
	 is a right-adjoint of the functor $W\otimes -:\SmoG\to\modG$.
	Consequently, we have:
	\begin{enumerate}
		\item
		$\Hom_{G-\st}(W,-)$ is left-exact;
		\item
		$\Hom_{G-\st}(W,-)$ takes relative injectives to injectives;
		\item
		$\Hom_{G-\st}(W,-)$ takes injective objects of $\modaG$
		to injective objects of $\SmoG$.
	\end{enumerate}
\end{lemma}

\begin{proof}
Let $X$ be a smooth $G$-module, and let $V$ and $W$ be continuous
$G$-modules.
We need to check that
$$
	\Hom_{G}(X,\Hom_{G-\st}(W,V))
	=
	\Hom_{G-\cts}(X\otimes W,V).
$$
We first note that $\Hom_{G-\st}(W,V)$ is the space of smooth
vectors in $\Hom_{\cts}(W,V)$.
Since all vectors in $X$ are smooth, we have
$$
	\Hom_{G}(X,\Hom_{G-\st}(W,V))
	=
	\Hom_{\Q_{p}}(X,\Hom_{\cts}(W,V))^{G},
$$
where we are regarding $\Hom_{\cts}(W,V)$ as an abstract vector space.
With our choice of topology on $X\otimes W$, we have:
\begin{eqnarray*}
	\Hom_{\Q_{p}}(X,\Hom_{\cts}(W,V))
	&=&
	\limp{F}\Hom_{\Q_{p}}(F,\Hom_{\cts}(W,V))
	=
	\limp{F}\Hom_{\cts}(F\otimes W,V)\\
	&=&
	\Hom_{\cts}(X\otimes W,V).
\end{eqnarray*}
Here $F$ runs over finite dimensional subspaces of $X$.
\end{proof}

We write $\Ext_{G-\st}^{\bullet}(W,-)$ for the relative derived functors
of the functor $\Hom_{G-\st}(W,-):\modG\to\SmoG$.
These are smooth $G$-modules,
and may be calculated as follows:
$$
	\Ext^{\bullet}_{G-\st}(W,V)
	=
	\limdr
	\Ext^{\bullet}_{G_{r}}(W,V).
$$

\subsection{Simplicial sheafs}
\def\modQp{\mathbf{mod}_{\Q_{p}}}
\def\modQpG{\mathbf{mod}_{\Q_{p}G}}

Let $\cY_{0}$ be an abstract finite simplicial complex.
For our purposes, this means $\cY_{0}$
 is a set of non-empty subsets (the simplexes) of a finite set (the set of vertices),
 such that if $A$ is in $\cY_{0}$ then every non-empty subset of $A$ is in
 $\cY_{0}$.
We regard $\cY_{0}$ as a topological space, in which the
closed sets are precisely the subcomplexes of $\cY_{0}$.
Given an abelian category $\bC$, we write $\sh(\bC/\cY_{0})$
for the category of $\bC$-valued sheafs on $\cY_{0}$.
Note that since $\cY_{0}$ is a finite topological space, it is not
necessary for $\bC$ to be closed under taking direct limits.
A standard argument shows that if $\bC$ has enough injectives
then so does $\sh(\bC/\cY_{0})$.
We shall write $\Gamma_{\cY_{0}}^{\bC}$ for the global sections
functor from $\sh(\bC/\cY_{0})$ to $\bC$.

\begin{lemma}
	Let $\bC$ be an abelian category.
	The functor $\Gamma_{\cY_{0}}^{\bC}:\sh(\bC/\cY_{0})\to\bC$
	 is a right adjoint of the constant sheaf functor $V\mapsto \underline V$.
	Consequently,
	\begin{enumerate}
		\item
		$\Gamma_{\cY_{0}}^{\bC}$ is left-exact;
		\item
		$\Gamma_{\cY_{0}}^{\bC}$ takes injectives in $\sh(\bC/\cY_{0})$
		 to injectives in $\bC$.
	\end{enumerate}
\end{lemma}

\begin{proof}
	This is standard.
\end{proof}

The categories $\sh(\SmoG/\cY_{0})$ and $\sh(\modaG/\cY_{0})$ have
 enough injectives.
However $\sh(\modG/\cY_{0})$ does not have enough injectives.
We get around this problem by exactly the same method as before.
A morphism $f:\cF\to\cG$ in $\sh(\modG/\cY_{0})$ is called strong
if for each point $y\in \cY_{0}$, the exact sequence
is stalks
$$
	0\to (\ker f)_{y} \to \cF_{y}\to \cG_{y} \to (\coker f)_{y} \to 0
$$
is split exact (i.e. chain homotopic to zero) in the category of
topological vector spaces.
One defines relatively injective objects and relatively injective
resolutions exactly as above.
One may use relatively injective resolutions to define relative derived functors.
It is important to check that these relative derived functors agree on the
full subcategory $\sh(\modaG/\cY_{0})$ with the usual derived functors.
To see this, note that every object of $\sh(\modaG/\cY_{0})$ has an injective
 resolution in $\sh(\modaG/\cY_{0})$
 by sheafs which are relatively injective in $\sh(\modG/\cY_{0})$.
As a consequence of this, we have:

\begin{lemma}
	If $\cF$ is injective in $\sh(\modaG/\cY_{0})$
	then $\cF$ is relatively injective in $\sh(\modG/\cY_{0})$.
\end{lemma}

\begin{proof}
As explained above, the relative derived functors of
 $\Hom_{\sh(\modG/\cY_{0})}(-,\cF)$ restrict to the 
 derived functors of $\Hom_{\sh(\modaG/\cY_{0})}(-,\cF)$.
\end{proof}

Let $W$ be an object of $\modG$.
We write $\cHom_{G-\st}(W,-)$ for the
functor from $\sh(\modG/\cY_{0})$ to $\sh(\SmoG/\cY_{0})$,
defined by
$$
	\cHom_{G-\st}(W,\cF)(U)
	=
	\cHom_{G-\st}(W,\cF(U)),
	\quad
	U\subset \cY_{0} \hbox{ open.}
$$
Similarly we write $-\otimes \uW$ for the functor
$\sh(\SmoG/\cY_{0})\to \sh(\modG/\cY_{0})$, defined by
$$
	(\cF\otimes \uW)(U)
	=
	\cF(U)\otimes W.
$$

\begin{lemma}
	Let $W$ be a continuous $G$-module.
	\begin{enumerate}
		\item
		The functor $\cHom_{G-\st}(W,-):\sh(\modG/\cY_{0}) \to \sh(\SmoG/\cY_{0})$
		is a right-adjoint of $-\otimes \uW$.
		\item
		$-\otimes \uW:\sh(\SmoG/\cY_{0})\to \sh(\modG/\cY_{0})$
		takes exact sequences to strong exact sequences.
		\item
		$\cHom_{G-\st}(W,-)$ is left exact and takes relative injectives to injectives.
	\end{enumerate}
\end{lemma}

\begin{proof}
Let $\cX$ be a sheaf of smooth $G$-modules on $\cY_{0}$
and let $\cV$ be a sheaf of continuous $G$-modules on $\cY_{0}$.
For part (1), we must show that
$$
	\Hom_{\sh(\SmoG/\cY_{0})}(\cX,\cHom_{G-\st}(W, \cV))
	=
	\Hom_{\sh(\modG/\cY_{0})}(\cX\otimes\uW, \cV).
$$
It is sufficient to verify that on each open set $U\subset \cY_{0}$, we have:
$$
	\Hom_{\SmoG}(\cX(U),\Hom_{G-\st}(W, \cV(U)))
	=
	\Hom_{\modG}(\cX(U)\otimes W, \cV(U)),
$$
and that these isomorphisms are compatible with restriction maps.
This follows from Lemma \ref{lemmasix}.
For part (2), recall that exactness of a sequence of sheafs
 is equivalent to exactness of the corresponding sequence of stalks
 at each point.
This allows us to reduce (2) to Lemma \ref{lemmaexacttensor}.
Part (3) follows from (1) and (2).
\end{proof}

\subsection{}
Let $Y_{0}$ be a geometric realization of the abstract simplicial complex $\cY_{0}$
and $Z_{0}$ a geometric realization of a subcomplex $\cZ_{0}$.
If $\bC$ is an abelian category, which is closed under direct limits,
then we may form the category $\sh(\bC/Y_{0})$.
The map $\pi:Y_{0}\to \cY_{0}$, which takes a point of $Y_{0}$ to the smallest
simplex containing that point, is continuous.

\begin{lemma}
	\label{lemmaten}
	Let $\bC$ be an abelian category with enough injectives,
	 which is closed under direct limits.
	Let $\cF$ be an object of $\sh(\bC/\cY_{0})$.
	If $\cF$ is a locally constant sheaf then
	$$
		H^{\bullet}(Y_{0},Z_{0},\cF)
		=
		H^{\bullet}(\cY_{0},\cZ_{0},\pi_{*}\cF).
	$$
\end{lemma}

\begin{proof}
	This follows from the spectral sequence of the map $\pi:Y_{0}\to \cY_{0}$:
	$$
		H^{p}(\cY_{0},(R^{q}\pi_{*})(\cF^{Z_{0}}))
		\implies
		H^{p+q}(Y_{0},Z_{0},\cF).
	$$
	We need to check that $(R^{q}\pi_{*})(\cF^{Z_{0}})=0$ for $q>0$.
	This amounts to checking for every $y\in\cY_{0}$ that
	$H^{q}(\pi^{-1}\cU,\pi^{-1}(\cU\cap\cZ_{0}),\cF)=0$,
	where $\cU$ is the intersection
	of the (finitely many) open subsets of $\cY_{0}$ containing $y$.
	Since both $\pi^{-1}\cU$ and $\pi^{-1}(\cU\cap\cZ_{0})$ are contractible (or empty)
	and $\cF$ is locally constant, the result follows.
\end{proof}

\begin{theorem}
	\label{myspecseq}
	Let $W$ be a continuous admissible $G$-module.
	Then there is a spectral sequence of smooth $G$-modules:
	$$
		\Ext^{p}_{G-\st}(W,H^{q}(\cY_{0},\cZ_{0},\scC(G,\Q_{p})))
		\implies
		\limdr
		H^{p+q}(Y_{r},Z_{r},\Hom_{\cts}(W,\Q_{p})).
	$$
\end{theorem}

\begin{proof}
Write $\cExt^{\bullet}_{G-\st}(W,-)$ for the derived functors
of $\cHom_{G-\st}(W,-)$, and write $\Gamma_{\cY_{0}}$ for the global section functor.
The functor $\Gamma_{\cY_{0}}$ is also left-exact and preserves injectives,
 and we write $H^{\bullet}(\cY_{0},-)$ for its derived functors.
There is an equality of functors from $\sh(\modG/\cY_{0})$ to $\SmoG$:
$$
	\Gamma_{\cY_{0}}\circ \cHom_{G-\st}(W,-)
	=
	\Hom_{G-\st}( W,-)\circ \Gamma_{\cY_{0}}.
$$
We write $\Gamma_{\sm}$ for this composition.
Hence there are Grothendieck spectral sequences:
$$
	\Ext^{p}_{G-\st}(W,H^{q}(\cY_{0},\scC(G,\Q_{p})^{\cZ_{0}}))
	\implies
	(R^{p+q}\Gamma_{\sm})(\scC(G,\Q_{p})^{\cZ_{0}}),
$$
$$
	H^{p}(\cY_{0},\cExt^{q}_{G-\st}(W, \scC(G,\Q_{p})^{\cZ_{0}}))
	\implies
	(R^{p+q}\Gamma_{\sm})(\scC(G,\Q_{p})^{\cZ_{0}}).
$$
The terms from the first spectral sequence are those
 from the statement of the theorem.
To prove the theorem, we must show that
 $(R^{p}\Gamma_{\smooth})(\scC(G,\Q_{p})^{\cZ_{0}})$
 is the classical cohomology with values in $\Hom_{\cts}(W,\Q_{p})$.

Since the stalks of $\scC(G,\Q_{p})^{\cZ_{0}}$ are injective objects of
$\modaG$, it follows that the extension groups
$\cExt^{q}_{G-\st}(W,\scC(G,\Q_{p})^{\cZ_{0}})$ are zero for $q>0$.
Hence the second spectral sequence degenerates,
and we have
$$
	(R^{p}\Gamma_{\sm})(\scC(G,\Q_{p})^{\cZ_{0}})
	=
	H^{p}(\cY_{0},\cHom_{G-\st}(W,\scC(G,\Q_{p})^{\cZ_{0}})).
$$
By Lemma \ref{lemmaten}, we have
$$
	H^{p}(\cY_{0},\cHom_{G-\st}(W,\scC(G,\Q_{p})^{\cZ_{0}}))
	=
	H^{p}(Y_{0},Z_{0}, \cHom_{G-\st}(W,\scC(G,\Q_{p}))).
$$
Since direct limits are exact, we have
$$
	H^{p}(Y_{0},Z_{0}, \cHom_{G-\st}(W,\scC(G,\Q_{p})))
	=
	\limdr
	H^{p}(Y_{0},Z_{0}, \cHom_{G_{r}}(W,\scC(G,\Q_{p}))).
$$
Let $f:Y_{r}\to Y_{0}$ be the projection map.
The spectral sequence of this map is:
$$
	H^{p}(Y_{0},Z_{0},H^{q}(G/G_{r}, \cHom_{G_{r}}(W,\scC(G_{r},\Q_{p}))))
	\implies
	H^{p+q}(Y_{r},Z_{r},\cHom_{G_{r}}(W,\scC(G_{r},\Q_{p}))).
$$
Since $G/G_{r}$ is a finite set, the spectral sequence degenerates and we have:
$$
	H^{p}(Y_{0},Z_{0},H^{0}(G/G_{r}, \cHom_{G_{r}}(W,\scC(G_{r},\Q_{p}))))
	=
	H^{p}(Y_{r},Z_{r},\cHom_{G_{r}}(W,\scC(G_{r},\Q_{p}))).
$$
As sheafs on $Y_{0}$ we have
$$
	H^{0}(G/G_{r}, \cHom_{G_{r}}(W,\scC(G_{r},\Q_{p})))
	=
	\cHom_{G_{r}}(W,\scC(G,\Q_{p})).
$$
By Frobenius reciprocity we have:
$$
	\cHom_{G_{r}}(W,\scC(G_{r},\Q_{p}))
	=
	\cHom_{\cts}(W,\Q_{p}).
$$
The result follows.
\end{proof}

\begin{remark}
Theorem \ref{myspecseq} gives Emerton's spectral sequence
 in the compactly supported case.
For the usual case, we take $Z_{0}$ to be empty.
Note that we have not required $W$ to be finite dimensional.
We also need no special properties of the simplicial complex $Y_{0}$.
If $W$ is finite dimensional, then we may replace
$\Ext^{\bullet}_{G-\st}$ by extension groups of the
representation of the Lie algebra of $G$ on the locally analytic vectors
(see Theorem 1.1.13 of \cite{emerton}).
\end{remark}

\section{Calculation of some cohomology groups}
\label{abstractsection}

In this section we again let $G$ be a compact $p$-adic analytic group,
and we have a neighbourhood filtration $G_{r}$
 by open normal subgroups.
We let $Y_{0}$ be a finite simplicial complex,
 together with a map for each connected component to $Y_{0}$,
 from the fundamental group of the connected component to $G$.
We let $\scC(G)$ be the corresponding local system of $G$-modules on $Y_{0}$.
We shall calculate $H^{\bullet}(Y_{0},\scC(G))$ in some cases.
In the case that $Y_{0}$ is connected, we write $\Gamma$ for the
fundamental group of $Y_{0}$.

\subsection{Induction}

For the purposes of the discussion here, a topological $G$-module will be
an abelian topological group $M$ with an action of $G$
 by a continuous map $G\times M\to M$.
We call $M$ smooth if it has the discrete topology.
For an open subgroup $H$ of finite index in $G$, we write $\ind_{H}^{G}$
for the induction functor from smooth $H$-modules to smooth $G$-modules,
or from topological $H$-modules to topological $G$-modules.

Suppose instead that $H$ is a closed subgroup, of possibly infinite index.
We shall describe two induction functors in this case.
Given a smooth $H$-module $V$, we define $\smoothind_{H}^{G}(V)$ to be the space
of smooth functions $f:G\to V$, such that $f(hg)=hf(g)$ for all $h\in H$.
The smooth induced module is a smooth $G$-module.

Given a continuous $H$-module $V$, we define $\hind_{H}^{G}V$ to be the space
of continuous functions $f:G\to V$ satisfying $f(hg)=hf(g)$ for all $h\in H$.
We give $\hind_{H}^{G}V$ the topology of uniform convergence.
It is a continuous $G$-module.
In fact $\smoothind$ is simply the restriction of $\hind$ to the category
of smooth modules, but it is convenient to distinguish between these two functors.

We now discuss how these three kinds of induction are related.

\begin{lemma}
	\label{dlimind}
	Let $H$ be a closed subgroup of $G$.
	Suppose we have an inductive system of $H$-modules $V_{r}$,
	such that for each $r$, $V_{r}$ is the inflation of a $H/(H\cap G_{r})$-module.
	By identifying $H/(H\cap G_{r})$ with $G_{r}H/G_{r}$, we also regard
	$V_{r}$ as a $G_{r}H$-module.
	Then we have an isomorphism of smooth $G$-modules:
	$$
		\limdr \ind_{G_{r}H}^{G} V_{r}
		=
		\smoothind_{H}^{G} \limdr V_{r}.
	$$
\end{lemma}

\begin{proof}
Let $x\in \limdr \ind_{G_{r}H}^{G} V_{r}$.
This means that for $r$ sufficiently large, $x$ is represented by
 a function $x_{r}:G\to \limdr V_{r}$, with the properties
\begin{enumerate}
	\item
	For $g\in G$ and $h\in H$, $x_{r}(hg)=h \cdot x_{r}(g)$;
	\item
	$x_{r}(gg')= x_{r}(g')$ for all $g\in G_{r}$ and $g\in G$.
\end{enumerate}
Only the second condition depends on $r$;
furthermore the union over all $r$ of such spaces of functions
 is the space of smooth functions with values in $\limdr V_{r}$
 satisfying the first condition.
\end{proof}

\begin{lemma}
	\label{indprojs}
	Let $H$ be a closed subgroup of $G$
	 and let $V_{s}$ be a projective system of smooth (discrete) $H$-modules.
	Then we have
	$$
		\limprojs \smoothind_{H}^{G}V_{s}
		=
		\hind_{H}^{G} \limprojs V_{s}.
	$$
\end{lemma}

\begin{proof}
This is easy to check.
\end{proof}

\begin{theorem}
	\label{excise}
	Let $H\subset G$ be an analytic subgroup.
	Assume that for each connected component of $Y_{0}$,
	the image in $G$ of the corresponding fundamental group
	is contained in $H$.
	Then there is an isomorphism of $G$-modules:
	$$
		H^{\bullet}(Y_{0},\scC(G))
		=
		\hind_{H}^{G}
		H^{\bullet}(Y_{0},\scC(H)).
	$$
\end{theorem}

\begin{proof}
We use Emerton's point of view.
The space $Y_{r}$ is given by
$$
	Y_{r}=\Gamma\backslash (Y^{\univ}\times G/G_{r}).
$$
There is a projection $Y_{r}\to (G_{r}H)\backslash G$ given by $\Gamma(y,g)\mapsto HG_{r}g$.
We write $X_{r}$ for the preimage of the identity coset $HG_{r}$.
This is given by
$$
	X_{r}
	=
	\Gamma\backslash (Y^{\univ}\times HG_{r}/G_{r}).
$$
Letting $H_{r}=H\cap G_{r}$, and identifying $HG_{r}/G_{r}$ with $H/H_{r}$,
we see that the sets $X_{r}$ form a projective system,
whose projective limit is equal to
$$
	X
	=
	\Gamma\backslash (Y^{\univ}\times H).
$$
Since the index $[G:G_{r}H]$ is finite, it is clear that
$$
	H^{\bullet}(Y_{r},\Z/p^{s})
	=
	\ind_{HG_{r}}^{G} H^{\bullet}(X_{r},\Z/p^{s}).
$$
Taking the limit over levels $G_{r}$, we obtain:
$$
	\limdr H^{\bullet}(Y_{r},\Z/p^{s})
	=
	\limdr
	\ind_{HG_{r}}^{G}
	H^{\bullet}(X_{r},\Z/p^{s}).
$$
From Lemma \ref{dlimind}, we have:
$$
	\limdr
	H^{\bullet}(Y_{r},\Z/p^{s})
	=
	\smoothind_{H}^{G}
	\limdr
	H^{\bullet}(X_{r},\Z/p^{s}).
$$
By Lemma \ref{indprojs} we have:
$$
	\limprojs
	\limdr
	H^{\bullet}(Y_{r},\Z/p^{s})
	=
	\hind_{H}^{G}
	\limprojs \limdr H^{\bullet}(X_{r},\Z/p^{s}).
$$
In Emerton's notation, this is:
$$
	\tH^{\bullet}(Y,\Z_{p})
	=
	\hind_{H}^{G}
	\tH^{\bullet}(X,\Z_{p}).
$$
The result follows from Theorems \ref{reinterpret1} and \ref{reinterpret2}.
\end{proof}

\begin{remark}
	The proof relies on the fact that $G$ and $H$ are $p$-adic analytic groups.
	The author does not know whether Theorem \ref{excise}
	 holds for profinite groups in general.
\end{remark}

\subsection{Cohomology of nil-manifolds}

For the next few results, it will be useful to recall
 that if a finite simplicial complex $Y_{0}$ is a $K(\Gamma,1)$-space
 then $\Gamma$ must be torsion-free and finitely presented,
 and $Y_{0}$ must be connected.

\begin{lemma}
	\label{abelian}
	Let $Y_{0}$ be a $K(\Gamma,1)$-space, where $\Gamma$ abelian.
	Assume that the map $\Gamma\to G$ identifies $G$ with the
	 $p$-adic completion of $\Gamma$.
	Then we have
	$$
		H^{n}(Y_{0},\scC(G))
		=
		\begin{cases}
			\Z_{p} & n=0,\\
			0 & n>0.
		\end{cases}
	$$
	The action of $G$ on these spaces is trivial.
\end{lemma}

\begin{proof}
We have $\Gamma=\Z^{N}$ and $G=\Z_{p}^{N}$ for some $N$,
and $Y_{0}$ is homotopic to a torus.
Letting $G_{r}=p^{r}\Z_{p}^{N}$, it follows that
$Y_{r}$ is naturally a $K(p^{r}\Z^{N},1)$ space for every $r$.
Hence by Emerton's formula for $\tH^{\bullet}(Y,\Z_{p})$,
 we have:
$$
	H^{n}(Y_{0},\scC(G))
	=
	\limprojs \limdr
	H^{n}(Y_{r},\Z/p^{s})
	=
	\limprojs \limdr
	H^{n}(p^{r}\Z^{N},\Z/p^{s}).
$$
On the other hand, $H^{n}(\Z^{N},\Z/p^{s})=\Hom(\wedge^{n}(\Z^{N}),\Z/p^{s})$.
This implies
$$
	\limdr
	H^{n}(p^{r}\Z^{N},\Z/p^{s})
	=
	\begin{cases}
		\Z/p^{s}
		&
		n=0,\\
		0
		&
		n>0.
	\end{cases}
$$
The result follows.
\end{proof}

It will be convenient to have a more precise version of this result:

\begin{lemma}
	\label{abstractdefect}
	Let $Y_{0}$ be a $K(\Gamma,1)$-space,
	 where $\Gamma$ is abelian;
	 let $\hat\Gamma$ be the $p$-adic completion of $\Gamma$.
	Suppose that the induced map $\hat\Gamma\to G$ is surjective,
	and let $\Delta$ be the kernel of this map.
	Then we have
	$$
		H^{n}(Y_{0},\scC(G))
		=
		\Hom_{\Z_{p}}(\wedge^{n}_{\Z_{p}}\Delta,\Z_{p}).
	$$
	The action of $G$ is trivial.
\end{lemma}

\begin{proof}
As in the proof of the previous result, we have:
$$
	H^{n}(Y_{0},\scC(G))
	=
	\limprojs \limdr
	H^{n}(\Gamma_{r},\Z/p^{s}),
$$
where $\Gamma_{r}$ is the preimage of $G_{r}$ in $\Gamma$.
Since $\Gamma_{r}$ is a finitely generated torsion-free abelian group, its
cohomology is as follows:
$$
	H^{n}(\Gamma_{r},\Z/p^{s})
	=
	\Hom(\wedge^{n}\Gamma_{r},\Z/p^{s}).
$$
This is clearly equal to $\Hom_{\Z_{p}}(\wedge^{n}_{\Z_{p}}\hat\Gamma_{r},\Z/p^{s})$,
where $\hat\Gamma_{r}$ is the preimage of $G_{r}$ in $\hat\Gamma$.
The kernel $\Delta$ is the intersection of the subgroups $\hat\Gamma_{r}$,
 and one easily checks that
$$
	\limdr
	\Hom_{\Z_{p}}(\wedge^{n}_{\Z_{p}}\hat\Gamma_{r},\Z/p^{s})
	=
	\Hom_{\Z_{p}}(\wedge^{n}_{\Z_{p}}\Delta,\Z/p^{s}).
$$
The result follows.
\end{proof}

\begin{theorem}
	\label{nilpotent}
	Let $Y_{0}$ be a $K(\Gamma,1)$ space,
	 where $\Gamma$ is a nilpotent group.
	Assume that the map $\Gamma\to G$ identifies $G$
	with the pro-$p$ completion of $\Gamma$.
	Then we have
	$$
		H^{n}(Y_{0},\scC(G))
		=
		\begin{cases}
			\Z_{p} & n=0,\\
			0 & n>0.
		\end{cases}
	$$
	The action of $G$ is trivial.
\end{theorem}

\begin{proof}
We consider the ascending central series:
$$
	1\subset
	\Gamma_{1}\subset
	\Gamma_{2}\subset
	\cdots \subset \Gamma_{N}=\Gamma,
$$
where $\Gamma_{i+1}/\Gamma_{i}$ is the centre of $\Gamma/\Gamma_{i}$.
We let $G_{i}$ be the pro-$p$ completion of $\Gamma_{i}$.
Each of the subquotients $\Gamma_{i+1}/\Gamma_{i}$ is a finitely
generated torsion-free abelian group, and its $p$-adic completion
is $G_{i+1}/G_{i}$.
We shall prove by induction on $t$ that
\begin{equation}
	\label{hypothesis}
	H^{n}(\Gamma_{t},\scC(G_{t}))
	=
	\begin{cases}
		\Z_{p}& n=0,\\
		0&n>0.
	\end{cases}
\end{equation}
If $t=1$ then $\Gamma_{t}$ is abelian,
 so the result follows from Lemma \ref{abelian}.
Assume that (\ref{hypothesis}) holds for $t-1$.
There is a Hochschild--Serre spectral sequence
$$
	H^{p}(\Gamma_{t} / \Gamma_{t-1}, H^{q}(\Gamma_{t-1}, \scC(G_{t})))
	\implies
	H^{p+q}(\Gamma_{t}, \scC(G_{t})).
$$
By Theorem \ref{excise}, we have 
$$
	H^{\bullet}(\Gamma_{t-1}, \scC(G_{t}))
	=
	\hind_{G_{t-1}}^{G_{t}} H^{\bullet}(\Gamma_{t-1},\scC(G_{t-1})),
$$
and hence by the inductive hypothesis, we have:
$$
	H^{q}(\Gamma_{t-1}, \scC(G_{t}))
	=
	\begin{cases}
		\scC( G_{t}/G_{t-1}) & q=0,\\
		0 & q>0.
	\end{cases}
$$
It follows that the spectral sequence has only one non-zero row,
so we have:
$$
	H^{\bullet}(\Gamma_{t}, \scC(G))
	=
	H^{\bullet}(\Gamma_{t} / \Gamma_{t-1}, \scC(G_{t}/G_{t-1})).
$$
On the other hand $G_{t}/G_{t-1}$ is the $p$-adic completion
of $\Gamma_{t}/\Gamma_{t-1}$, so Lemma \ref{abelian} gives:
$$
	H^{n}(\Gamma_{t} / \Gamma_{t-1}, \scC(G_{t}/G_{t-1}))
	=
	\begin{cases}
		\Z_{p} & n=0,\\
		0 & n>0.
	\end{cases}
$$
\end{proof}

\begin{theorem}
	\label{normalnilpotent}
	Let $Y_{0}$ be a $K(\Gamma,1)$ space.
	Let $N$ be an analytic, normal, nilpotent subgroup of $G$,
	and let $\Gamma_{N}$ be the preimage of $\Gamma$ in $N$.
	Assume that $N$ is the pro-$p$ completion of $\Gamma_{N}$.
	Then we have as $G$-modules:
	$$
		H^{\bullet}(\Gamma, \scC(G))
		=
		H^{\bullet}(\Gamma/\Gamma_{N},\scC(G/N)).
	$$
\end{theorem}

\begin{proof}
There is a spectral sequence
$$
	H^{p}(\Gamma/\Gamma_{N},H^{q}(\Gamma_{N},\scC(G))
	\implies
	H^{p+q}(\Gamma,\scC(G)).
$$
By Theorem \ref{excise}, we have
$$
	H^{\bullet}(\Gamma_{N},\scC(G))
	=
	\hind_{N}^{G}
	H^{\bullet}(\Gamma_{N},\scC(N)),
$$
and hence by the previous theorem we have
$$
	H^{q}(\Gamma_{N},\scC(G))
	=
	\begin{cases}
		\scC(G/N) & q=0,\\
		0 & q>0.
	\end{cases}
$$
The result follows.
\end{proof}

\subsection{First applications to arithmetic quotients}
\label{concretesection}

We begin now to apply our results to arithmetic quotients
as described in the introduction.
For simplicity we assume now that $k=\Q$.
Our methods extend to the general case; however the results are more difficult to state.
This is because one is forced to make systematic use of Lemma \ref{abstractdefect}
in place of Lemma \ref{abelian} and Theorem \ref{nilpotent}.
The results stated below remain correct if one replaces $\Q$
 by a number field $k$, in which $\gp$ is the only prime of $k$ above $p$.

For a linear algebraic group $\bG$ over $\Q$, and a compact open subgroup
$K_{f}\subset\bG(\A_{f})$, we shall use the notation
$$
	Y_{\bG}(K_{f})
	=
	\bG(\Q)\backslash \bG(\A) / K_{\infty}^{\circ}K_{f},
$$
where $K_{\infty}^{\circ}$ is a fixed maximal compact connected
subgroup of $\bG(\R)$.

\subsubsection{Unipotent groups}

We recall that a group $\bG/\Q$ is said to satisfy strong approximation
if $\bG(\Q)\bG(\R)$ is dense in $\bG(\A)$.

\begin{lemma}
	\label{strongapproximation}
	If $\bG$ satisfies strong approximation then $Y_{\bG}(K_{f})$
	is connected.
	If, in addition, $\Gamma(K_{f})$ is torsion-free then $Y(K_{f})$
	is a $K(\Gamma(K_{f}),1)$ space.
\end{lemma}

\begin{proof}
This is easy and well-known.
\end{proof}

\begin{theorem}
	\label{unipotent}
	If $\bN/\Q$ is a unipotent group and
	$K^{p}$ (resp. $K_{p}$) is a compact open subgroup of
	$\bN(\A_{f}^{p})$ (resp. $\bN(\Q_{p})$) then
	$$
		H^{n}(Y_{\bN}(K_{p}K^{p}),\scC(K_{p}))
		=
		\begin{cases}
			\Z_{p} & n=0,\\
			0 & n>0.
		\end{cases}
	$$
\end{theorem}

\begin{proof}
It is known that strong approximation holds for unipotent groups.
Furthermore $\bN(\A)$ is torsion-free.
It follows from Lemma \ref{strongapproximation} that $Y_{\bN}(K_{f})$ is a
$K(\Gamma_{\bN}(K_{f}),1)$-space.
The arithmetic group $\Gamma_{\bN}(K_{f})$ is nilpotent.
The pro-$p$ completion of $\Gamma(K_{f})$ injects into $K_{p}$.
By strong approximation, its image is the whole of $K_{p}$.
Thus $K_{p}$ is the pro-$p$ completion of $\Gamma_{\bN}(K_{p})$.
The result follows from Theorem \ref{nilpotent}.
\end{proof}

\subsubsection{Unipotent radicals}

Let $\bP$ be an arbitrary linear algebraic group over $\Q$.
In applications, $\bP$ will be a parabolic subgroup.
We let $\bN$ be the unipotent radical of $\bP$
and we choose a Levi component $\bL$.
A maximal compact subgroup $K_{\bL,\infty}$ of $\bL(\R)$
is also maximal compact in $\bP(\R)$.
We let $K_{f,\bP}=K_{p,\bP}K^{p}_{\bP}$ be a compact open
 subgroup of $\bP(\A_{f})$ with $K_{p,\bP}$ torsion-free.
We also let $K_{f,\bN}=K_{f}\cap \bN$, and similarly $K_{p,\bN}$.
The subgroup $K_{f,\bN}$ is normal in $K_{f,\bP}$,
and we define $K_{f,\bL}=K_{f,\bP}/K_{f,\bN}$.
We may regard $K_{f,\bL}$ as a subgroup of $\bL(\A_{f})$; however it may
be larger than $K_{f}\cap \bL$.

\begin{theorem}
	\label{parabolic}
	With the notation just described, there is an
	isomorphism of $K_{p,\bP}$-modules:
	$$
		H^{\bullet}_{\ast}(Y_{\bP}(K_{f,\bP}),\scC(K_{p,\bP}))
		=
		H^{\bullet}_{\ast}(Y_{\bL}(K_{f,\bL}),\scC(K_{p,\bL})),
	$$
	where $\ast$ is either the empty symbol or ``$\cpct$''.
	The action of $K_{p,\bP}$
	on $H^{\bullet}_{\ast}(Y_{\bL}(K_{f,\bL}),\scC(K_{p,\bL}))$
	is by right translation on $K_{p,\bL}= K_{p,\bP}/K_{p,\bN}$.
\end{theorem}

\begin{proof}
As $\bN$ is a normal subgroup of $\bG$, there is a projection map
$$
	f:Y_{\bP}(K_{f,\bP}) \to Y_{\bL}(K_{f,\bL}).
$$
This map is a fibre bundle with fibre $Y_{\bN}(K_{f,\bN})$.
The fibre is compact, so in both the compactly supported and also the usual case,
the spectral sequence is:
$$
	H^{p}_{\ast}(Y_{\bL}(K_{f,\bL}),H^{q}(Y_{\bN}(K_{f,\bN}),\scC(K_{p,\bP})))
	\implies
	H^{p+q}_{\ast}(Y_{\bP}(K_{f,\bP}),\scC(K_{p,\bP})).
$$
By Theorem \ref{excise} and Theorem \ref{unipotent}
we have
$$
	H^{q}(Y_{\bN}(K_{f,\bN}),\scC(K_{p,\bP}))
	=
	\begin{cases}
		\hind_{K_{p,\bN}}^{K_{p,\bP}} \Z_{p} & q=0,\\
		0 & q>0.
	\end{cases}
$$
Note that there is a canonical isomorphism
$$
	\hind_{K_{p,\bN}}^{K_{p,\bP}} \Z_{p}
	=
	\scC(K_{p,\bL}).
$$
The result follows.
\end{proof}

\subsubsection{Tori and Leopoldt's Conjecture}

Let $\bT$ be a torus over $\Q$,
and let $K_{f}=K^{p}K_{p}$ be a compact open subgroup
of $\bT(\A_{f})$.
We let $\Gamma(K_{f})=\bT(\Q) \cap (\bT(\R)^{\circ}K_{f})$,
and we let $\hat \Gamma(K_{f})$ be the $p$-adic completion of
$\Gamma(K_{f})$.
There is a map $\hat\Gamma(K_{f})\to K_{p}$, and we write
$\Delta(\bT,K_{f})$ for the kernel of this map.
The rank $\gd(\bT)=\rank_{\Z_{p}}\Delta (\bT,K_{f})$ does not depend on
$K_{f}$, and is called the \emph{Leopoldt defect of $\bT$}.

\begin{conjecture}[Leopoldt's Conjecture; see 10.3.5, 10.3.6 of \cite{NeukirchSchmidtWingberg}]
	$\gd(\bT)=0$.
\end{conjecture}

This conjecture is usually made for tori of the form $\Rest^{k}_{\Q}\GL_{1}/k$,
where $k$ is a number field.
The apparently more general conjecture given above is in fact equivalent to
the conjecture for all tori of the form $\Rest^{k}_{\Q}\GL_{1}/k$.
This follows easily from the classification of tori (see for example \S3.8.12 of \cite{borelbook}) together with the Artin induction theorem (\S15.4 of \cite{curtisreiner}).
Leopoldt's conjecture is known to hold for tori
 which split over an abelian extension of $\Q$.

\begin{theorem}
	\label{torus}
	Let $\bT$ be a torus over $\Q$ and let $K_{f}=K_{p}K^{p}$
	be chosen so that $\Gamma(K_{f})$ is torsion-free.
	Then we have:
	$$
		H^{n}(Y_{\bT}(K_{f}),\scC(K_{p}))
		=
		\ind_{G}^{K_{p}}\left(
		\Hom_{\Z_{p}}(\wedge^{n}_{\Z_{p}} \Delta, \Z_{p})
		\otimes
		\Z_{p}[\pi_{0}]
		\right),
	$$
	where $G$ (resp. $\Delta$) is the image
	(resp. the kernel) of the map $\hat\Gamma(K_{f})\to K_{p}$
	and $\pi_{0}$ is the group of components of the Lie group
	$Y_{\bT}(K_{f})$.
	The action of $G$ on $\Hom(\wedge^{n} \Delta, \Z_{p})\otimes\Z_{p}[\pi_{0}]$
	 is trivial.
	For compactly supported cohomology we have:
	$$
		H^{n}_{\cpct}(Y_{\bT}(K_{f}),\scC(K_{p}))
		=
		\begin{cases}
			\ind_{G}^{K_{p}}\left(
			\Hom_{\Z_{p}}(\wedge^{n-r}_{\Z_{p}} \Delta, \Z_{p})
			\otimes
			\Z_{p}[\pi_{0}]
			\right)
			& n\ge r,\\
			0 & n<r.
		\end{cases}
	$$
	where $r=\rank_{\Q}\bT$.
\end{theorem}

\begin{proof}
By Theorem \ref{excise}, it is sufficient to calculate the groups
$$
	H^{\bullet}_{\ast}(Y_{\bT}(K_{f}),\scC(G)).
$$
The space $Y_{\bT}(K_{f})$ is a Lie group.
We shall write $Y_{0}$ for its identity component.
The group $\pi_{0}$ of components is finite,
 and by the K\"unneth formula we have (as $G$-modules):
$$
	H^{\bullet}_{\ast}(Y_{\bT}(K_{f}),\scC(G))
	=
	H^{\bullet}_{\ast}(Y_{0},\scC(G))\otimes \Z_{p}[\pi_{0}].
$$
The identity component $Y_{0}$ is homeomorphic to $V/\Gamma$,
 where $V$ is a vector space over $\R$
 and $\Gamma$ is a lattice, which we identify with $\Gamma(K_{f})$.
Our use of the word ``lattice'' does not imply that $\Gamma$ spans $V$.
In general, $\dim(V)=\rank_{\R}(\bT)$, whereas (by Dirichlet's unit theorem)
 $\Gamma$ spans a subspace $V_{0}$ of rank $\rank_{\R}(\bT)-\rank_{\Q}(\bT)$.

In any case, $Y_{0}$ is a $K(\Gamma,1)$-space,
 so by Lemma \ref{abstractdefect} we have:
$$
	H^{\bullet}(Y_{0},\scC(G))
	=
	\Hom_{\Z_{p}}(\wedge_{\Z_{p}}^{\bullet}\Delta,\Z_{p}).
$$
The first part of the theorem follows.

To calculate the compactly supported cohomology,
we choose a decomposition $V=V_{0}\times V_{1}$ as real vector spaces.
The space $V_{0}/\Gamma$ is a compact $K(\Gamma,1)$ space,
so we have:
$$
	H^{n}_{\cpct}(V_{0}/\Gamma,\scC(G))
	=
	H^{n}(Y_{0},\scC(G)).
$$
Furthermore, since $V_{1}$ is an $r$-dimensional Euclidean space, we have:
$$
	H^{n}_{\cpct}(V_{1},\scC(G))
	=
	\begin{cases}
		\scC(G) & n=r,\\
		0 & n\ne r.
	\end{cases}
$$
The result follows from the spectral sequence of the map $Y_{0}\to V_{0}/\Gamma$.
\end{proof}

\begin{corollary}
	\label{leopoldt}
	Let $\bT$ be a torus over $\Q$ and let $\gd(\bT)$ be the Leopoldt
	defect of $\bT$.
	Assume $\Gamma(K_{f})$ is torsion-free.
	Then $H^{n}(Y_{\bT}(K_{f}),\scC(K_{p}))$ is non-zero
	if and only if $n\le \gd(\bT)$.
	In particular the following are equivalent:
	\begin{enumerate}
		\item
		Leopoldt's conjecture holds for $\bT$;
		\item
		$H^{1}(Y_{\bT}(K_{f}),\scC(K_{p}))=0$;
		\item
		$H^{n}(Y_{\bT}(K_{f}),\scC(K_{p}))=0$ for all $n>0$.
	\end{enumerate}
\end{corollary}

\subsubsection{Universal covers of semi-simple groups}

\begin{theorem}
	\label{universalcover}
	Let $\bG/\Q$ be the (algebraic) universal cover of
	a semi-simple group $\bH/\Q$.
	For any tame level $K^{p}_{\bH}$ in $\bH$,
	the group $\tH^{n}(K^{p}_{\bH},\Z_{p})$ is a sum of finitely many
	copies of $\tH^{n}(K^{p}_{\bG},\Z_{p})$, where $K^{p}_{\bG}$
	is the preimage of $K^{p}_{\bH}$ in $\bG$.
\end{theorem}

\begin{proof}
We have a short exact sequence of algebraic groups over $\Q$:
$$
	1
	\to
	\bF
	\to
	\bG
	\to
	\bH
	\to
	1,
$$
where $\bF$ is finite and abelian.
This gives rise to a long exact sequence in Galois cohomology.
$$
	1
	\to
	\bF(\Q_{p})
	\to
	\bG(\Q_{p})
	\to
	\bH(\Q_{p})
	\to
	H^{1}(\Q_{p},\bF)
	\to
	1.
$$
The triviality of the final term is a theorem of Kneser
(see \cite{kneser-H1part1} and also \S III.3.1 of \cite{serre-galoiscohomology}).

The group $H^{1}(\Q_{p},\bF)$ is finite by local class field theory.
We may choose a level $K_{p,\bG}\subset \bG(\Q_{p})$
small enough so that it is torsion-free.
In particular its intersection with $\bF(\Q_{p})$ is trivial.
We let $K_{p,\bH}$ be the image of $K_{p,\bG}$ in $\bH(\Q_{p})$.
This is a compact open subgroup of $\bH(\Q_{p})$ and is isomorphic to
$K_{p,\bG}$.
Replacing $K_{p,\bH}$ and $K_{p,\bG}$ by subgroups
if necessary, we may ensure that for every element $h\in \bH(\Q_{p})$,
the subgroup $h K_{p,\bH}h^{-1}$ is the bijective image of a torsion-free
subgroup of $\bG(\Q_{p})$.
This follows from the finiteness of $H^{1}(\Q_{p},\bF)$.

Choose a tame level $K^{p}_{\bH}$ in $\bH$, and define
$K^{p}_{\bG}$ to be its preimage in $\bG(\A_{f})$.
We let $K_{f,\bH}=K_{p,\bH}K^{p}_{\bH}$, and similarly for $\bG$.

The arithmetic quotient $Y_{\bH}(K_{f})$ is in
general not connected. Its connected components are
indexed by the finite set
$$
	\bH(\Q)\backslash\bH(\A)/\bH(\R)^{\circ}K_{f}.
$$
More precisely, the double coset containing an element $h\in\bH(\A)$
corresponds to a component
$$
	Y_{0}(h)
	=
	\Gamma\backslash \bH(\R)^{\circ}/K_{\infty,\bH}^{\circ},
	\qquad
	\Gamma
	=
	\Gamma(h^{-1}K_{f,\bH}h).
$$
Since $h^{-1}K_{p,\bH}h$ is torsion-free,
$\Gamma$ must be torsion-free, so $Y(h)$ is a $K(\Gamma,1)$-space.
The action of $\Gamma$ on $\scC(K_{p})$
 is the composition of the usual left-translation,
 with conjugation by $h$.
This implies:
$$
	H^{\bullet}(Y_{0}(h),\scC(K_{p,\bH}))
	=
	H^{\bullet}(Y_{0}(h),\scC(h K_{p,\bH}h^{-1})),
$$
where the action of $\Gamma$ on $\scC(h K_{p}h^{-1})$
is by left translation.

Choose a compact open subgroup $G(h)\subset\bG(\Q_{p})$ such that
$h K_{p,\bH}h^{-1}$ is the bijective image of $G(h)$.
By the strong approximation theorem,
$\Gamma$ is the bijective image of the congruence subgroup
$\Gamma_{\bG}(K_{\bG}^{p}G(h))$ of $\bG$, and there is a natural bijection:
$$
	Y_{0}(h)
	=
	Y_{\bG}(K^{p}_{\bG}G(h)).
$$
This implies
$$
	H^{\bullet}(Y_{0}(h),\scC(h K_{p,\bH}h^{-1}))
	=
	H^{\bullet}(Y_{\bG}(K^{p}_{\bG}G(h)),\scC(G(h))).
$$
The result follows.
\end{proof}

\section{Vanishing Theorems}
\label{borelserresection}

We return in this section to the notation of the introduction.

\begin{theorem}
	\label{borelserreboundary}
	Let $\bG$ be a reductive group over $\Q$.
	Assume that the maximal split torus in the centre of $\bG$ is trivial.
	Then $\cd_{\partial}(\bG,\Z_{p})\le D(\bG,\Z_{p})$
	and $\cd_{\partial}(\bG,\Q_{p})\le D(\bG,\Q_{p})$.
\end{theorem}

\begin{proof}
We prove the result in the case of coefficients in $\Z_{p}$.
The proof for $\Q_{p}$ is the same.

Choose a level $K_{f}=K_{p}K^{p}$,
 with $K_{p}$ small enough so that it is torsion-free.
We have (as $K_{p}$-modules):
$$
	\tH^{\bullet}_{\partial}(K^{p},\Z_{p})
	=
	H^{\bullet}(\partial Y(K^{f})^{\BS},\scC(K_{p})).
$$
We must prove that $\tH^{n}_{\partial}(K^{p},\Z_{p})=0$ for $n>D(\bG,\Z_{p})$.
It is clearly sufficient to prove, for each connected component
 $Y_{0}$ of $Y(K^{f})$,
 that $H^{n}(\partial Y_{0}^{\BS},\scC(K_{p}))=0$ for $n>D(\bG,\Z_{p})$.

Fix a component $Y_{0}$ of $Y(K_{f})$.
We have, for a congruence subgroup $\Gamma$ of $\bG$,
$$
	Y_{0}
	=
	\Gamma\backslash \bG(\R)^{\circ}/K_{\infty}^{\circ}.
$$
Replacing $K_{f}$ by a conjugate if necessary,
 we may assume that $\Gamma=\Gamma(K_{f})$.
Hence $\Gamma$ is torsion-free and $Y_{0}$ is a $K(\Gamma,1)$-space.

We recall the structure of the Borel--Serre boundary of $Y_{0}$.
Let $\bP_{1},\ldots,\bP_{N}$ be the $\Gamma$-conjugacy classes
of proper parabolic $\Q$-subgroups of $\bG$.
For each $\bP_{t}$, we have a Levi decomposition:
$$
	\bP_{t}
	=
	\bL_{t}\ltimes \bN_{t},
$$
 where $\bN_{t}$ is the unipotent radical of $\bP_{t}$
 and $\bL_{t}$ is a Levi component of $\bP_{t}$.
We further decompose $\bL_{t}=\bA_{t}\bM_{t}$,
 where $\bA_{t}$ is the maximal $\Q$-split torus in the centre of $\bL_{t}$
 and $\bM_{t}$ is the intersection of the kernels of the
 homomorphisms $\bL_{t}\to\GL_{1}/\Q$.
We shall choose $\bP_{t}$ and $\bL_{t}$ in such a way that
 $K_{\infty,\bM_{t}}=\bP_{t}(\R)\cap K_{\infty}$
 is a maximal compact subgroup of $\bM_{t}(\R)$.
We define a subgroup $\bQ_{t}=\bM_{t}\ltimes \bN_{t}$.
For each parabolic subgroup $\bP_{t}$ we have a boundary component:
$$
	Z_{t}
	=
	(\Gamma\cap \bQ_{t})
	\backslash
	\bQ_{t}(\R)^{\circ}/K_{\infty,\bM_{t}}^{\circ}.
$$
Note that $Z_{t}$ is a connected component of
the arithmetic quotient:
$$	
	Y_{\bQ_{t}}(Q_{t,f}),
	\qquad
	Q_{t,f}=K_{f}\cap \bQ_{t}(\A_{f}).
$$
For $n>D(\bG,\Z_{p})$, we know by Theorems
 \ref{excise} and \ref{parabolic}
 that $\tH^{n}_{\cpct}(Y_{\bQ_{t}}(Q_{t,f}),\scC(K_{p}))=0$.
It follows that
\begin{equation}
	\label{useful}
	\tH^{n}_{\cpct}(Z_{t},\scC(K_{p}))
	=
	0
	\hbox{ for }
	n>D(\bG,\Z_{p}).
\end{equation}

As a set, the Borel--Serre boundary of $Y_{0}$ is given by:
$$
	\partial Y_{0}^{\BS}
	=
	Z_{1}\cup \cdots \cup Z_{N}.
$$
Each component $Z_{t}$ in turn has a Borel--Serre compactification
 $Z_{t}^{\BS}$, which may be identified with
 the closure of $Z_{t}$ in $Y_{0}^{\BS}$.
The boundary $\partial Z_{t}^{\BS}$ is the union of those components $Z_{u}$
for which $\bP_{u}$ is $\Gamma$-conjugate to a proper subgroup of $\bP_{t}$.
We shall order the subgroups $\bP_{t}$ so that
 if $\bP_{t}$ is $\Gamma$-conjugate to a subgroup of $\bP_{u}$
 then $t<u$.
With this choice of ordering we have a filtration on $\partial Y_{0}^{\BS}$
by compact subspaces:
$$
	A_{1}\subset A_{2}\subset \cdots \subset A_{N}=\partial Y_{0}^{\BS},
$$
$$
	A_{t}
	=
	\bigcup_{u\le t}
	Z_{u}
	=
	\bigcup_{j\le i}
	Z_{u}^{\BS}.
$$
Let $n>D(\bG)$. We shall prove by induction on $t$
 that $H^{n}(A_{t},\scC(K_{p}))=0$.

In the case $t=1$, we have $A_{1}=Z_{1}$.
The subgroup $\bP_{1}$ is a minimal parabolic subgroup.
This implies that the rational rank of $\bM_{1}$ is zero.
It follows that $Z_{1}$ is compact, so by (\ref{useful})
we have:
$$
	H^{n}(A_{1},\scC(K_{p}))
	=
	0
	\hbox{ for }
	n>D(\bG,\Z_{p}).
$$

Assume that $H^{n}(A_{t-1},\scC(K_{p}))=0$ for $n>D(\bG)$
and consider the exact sequence
$$
	\to
	H^{n}(A_{t},A_{t-1},\scC(K_{p})) \to 
	H^{n}(A_{t},\scC(K_{p})) \to 
	H^{n}(A_{t-1},\scC(K_{p})) \to.
$$
By assumption, the last term in the sequence is zero.
For the inductive step, we shall show that the first term is zero.
By the excision axiom, we have
$$
	H^{n}(A_{t},A_{t-1},\scC(K_{p}))
	=
	H^{n}(Z_{t}^{\BS},\partial Z_{t}^{\BS},\scC(K_{p})).
$$
This is equal to $H^{n}_{\cpct}(Z_{t},\scC(K_{p}))$,
which by (\ref{useful}) is zero for $n>D(\bG,\Z_{p})$.
\end{proof}

\begin{corollary}
	\label{equality}
	Let $\bG/\Q$ be reductive and suppose the maximal $\Q$-split torus
	 in the centre of $\bG$ is trivial.
	\begin{enumerate}
		\item
		The map $\tH^{n}_{\cpct}(K^{p},\Q_{p})\to\tH^{n}(K^{p},\Q_{p})$
		is an isomorphism for $n>D(\bG,\Q_{p})+1$ and is
		surjective for $n=D(\bG,\Q_{p})+1$.
		\item
		The map $\tH^{n}_{\cpct}(K^{p},\Z_{p})\to\tH^{n}(K^{p}, \Z_{p})$
		is an isomorphism for $n>D(\bG, \Z_{p})+1$ and is
		surjective for $n=D(\bG, \Z_{p})+1$.
	\end{enumerate}
\end{corollary}

For a linear algebraic group $\bG$, we shall write $X_{\bG}$ for the
symmetric space $\bG(\R)/K_{\infty,\bG}$.
Recall the following:

\begin{theorem}[\cite{borel-serre}]
	If $\bG$ is reductive and the maximal split torus in the centre of
	$\bG$ is trivial, then the virtual cohomological dimension of
	arithmetic subgroups of $\bG$ is $\dim X_{\bG}-\rank_{\Q}\bG$.
\end{theorem}

As a result, we clearly have for such groups
$\cd(\bG,\Z_{p})\le \dim X_{\bG}-\rank_{\bG}\bG$.
We now prove the same bound for $\cd_{\cpct}$.

\begin{theorem}
	\label{vanishingtheorem}
	Let $\bG/\Q$ be reductive and suppose the maximal $\Q$-split torus
	 in the centre of $\bG$ is trivial.
	Then
	$$
		\cd_{\cpct}(\bG,\Z_{p})
		\le
		\dim X_{\bG}-\rank_{\Q}\bG.
	$$
\end{theorem}

\begin{proof}
	We prove this by induction on the rational rank of $\bG$.
	If $\bG$ has rank zero, then the result is trivial.
	Suppose the result is true for all groups of smaller rank then $\bG$.
	Let $\bP=\bM\bA\bN$ be a proper parabolic subgroup of $\bG$.
	By the Iwasawa decomposition,
	 we have a parametrization of the symmetric space $X_{\bG}$
	 as follows:
	$$
		X_{\bG}
		=
		X_{\bM}\times \bA(\R)^{\circ}\times \bN(\R).
	$$
	This implies $\dim X_{\bG}=\dim X_{\bM}+\rank_{\Q}\bA+\dim\bN$.
	By the inductive hypothesis, we have
	$$
		\cd_{\cpct}(\bM,\Z_{p})
		\le
		\dim X_{\bM}-\rank_{\Q}\bM.
	$$
	On the other hand $\rank_{\Q}\bG=\rank_{\Q}\bA+\rank_{\Q}\bM$.
	Combining these facts, we find that
	$$
		\cd_{\cpct}(\bM,\Z_{p})
		\le
		\dim X_{\bG}-\rank_{\Q}\bG-\dim \bN.
	$$
	In particular we have $D(\bG) <\dim X_{\bG}-\rank_{\Q}\bG$.
	By Corollary \ref{equality}, we have for $n>\dim X_{\bG}-\rank_{\Q}\bG$:
	$$
		\tH^{n}_{\cpct}(\bG,\Z_{p})
		=
		\tH^{n}(\bG,\Z_{p}).
	$$
	However, the right hand side is zero, since $n$ is larger than $\cd(\bG,\Z_{p})$.
\end{proof}

\begin{proposition}
	\label{topdimension}
	Let $\bG/\Q$ be a reductive group, such that the centre of $\bG$
	has rational rank zero.
	If the real rank of $\bG$ is positive then
	$$
		\cd(\bG,\Z_{p}),\cd_{\cpct}(\bG,\Z_{p})
		\le
		\dim X_{\bG}-1.
	$$
\end{proposition}

\begin{proof}
	If $\bG$ has positive rational rank, then this
	follows from our previous vanishing theorem.
	We therefore assume that $\bG$ has rational rank zero.
	In this case the arithmetic quotients $Y_{\bG}(K_{f})$ are
	compact, so we have $\cd(\bG,-)=\cd_{\cpct}(\bG,-)$.
	Choose a level $K_{f}=K^{p}K_{p}$, where $K_{p}$ is torsion-free.
	It is sufficient to prove that
	$$
		H^{d}(Y(K_{f}),\scC(K_{p}))=0,
		\qquad
		d=\dim X_{\bG}.
	$$
	Choose a component $Y_{0}$ of $Y(K_{f})$
	and let $\Gamma$ be the fundamental group of $Y_{0}$.
	We shall write $G$ for the closure of the image of $\Gamma$ in $K_{p}$.
	We have by Theorem \ref{excise}:
	$$
		H^{d}(Y_{0},\scC(K_{p}))
		=
		\hind_{G}^{K_{p}}
		H^{d}(Y_{0},\scC(G)).
	$$
	It is therefore sufficient to prove that $H^{d}(Y_{0},\scC(G))=0$.
	Choose a filtration
	$$
		G=G_{0}\supset G_{1}\supset \cdots,
	$$
	and for each $r$ let $Y_{r}$ be the corresponding cover of $Y_{0}$.
	Since the image of $\Gamma$ is dense in $G$,
	it follows that $Y_{r}$ is connected for every $r$.
	We therefore have $H^{d}(Y_{r},\Z/p^{s})=\Z/p^{s}$.
	For $r<t$ the map $H^{d}(Y_{r},\Z/p^{s})\to H^{d}(Y_{t},\Z/p^{s})$
	is multiplication by the index $[G:G_{r}]$.
	Since $G$ is locally a pro-$p$ group, it follows that
	$$
		\limdr H^{d}(Y_{r},\Z/p^{s})
		=
		0.
	$$
	Using Emerton's formula, we have:
	$$
		H^{d}(Y_{0},\scC(G))
		=
		\tH^{d}(Y_{0},\Z_{p})
		=
		\limprojs\limdr H^{d}(Y_{r},\Z/p^{s})
		=
		0.
	$$
\end{proof}

Recall that the congruence kernel $\Cong(\bG)$ of $\bG$
 is defined to be the kernel of the map
$$
	\limp{\Gamma \ arithmetic}
	\bG(\Q)/\Gamma
	\to
	\limp{K_{f}}
	\bG(\Q)/\Gamma(K_{f}).
$$
This measures the difference between the set of arithmetic subgroups
of $\bG$ and the set of congruence subgroups.
If $\bG$ is not simply connected, then the congruence kernel is
infinite.
Serre made the following conjecture:

\begin{conjecture}[\cite{serre-sl2}]
	Let $\bG/\Q$ be simple and simply connected, and assume that the real rank
	of $\bG$ is positive.
	Then $\Cong(\bG)$ is finite if and only if the real rank of $\bG$ is
	at least $2$.
\end{conjecture}

\begin{proposition}
	\label{nextdimensiondown}
	Let $\bG/\Q$ be a semi-simple group of positive real rank.
	If the congruence kernel of the universal cover of $\bG$ is finite
	 then
	$$
		\cd(\bG,\Z_{p}),\cd_{\cpct}(\bG,\Z_{p})
		\le
		\dim X_{\bG}-2.
	$$
	In particular this holds for groups of real rank $\ge 2$ satisfying Serre's
	conjecture.
\end{proposition}

\begin{proof}
	By Theorem \ref{universalcover} it is sufficient to prove this
	 for simply connected groups,
	 so we assume that $\bG$ is simply connected.
	It follows that the congruence kernel of $\bG$ is finite
	and $\bG$ satisfies strong approximation.
	Let $d=\dim X_{\bG}$.
	Proposition \ref{topdimension} shows that $\tH_{\cpct}^{d}(\bG,\Z_{p})=0$.
	To prove the result in the compactly supported case, we must
	show that $\tH^{d-1}_{\cpct}(\bG,\Z_{p})=0$.
	Choose a level $K_{f}=K^{p}K_{p}$ with $K_{p}$ torsion-free.
	The arithmetic quotient $Y(K_{f})$ is connected with fundamental
	group $\Gamma=\Gamma(K_{f})$.
	By Poincar\'e duality, we have
	$$
		H^{d-1}_{\cpct}(Y(K_{f}),\Z/p^{s})
		=
		H_{1}(\Gamma,\Z/p^{s})
		=
		\Gamma^{\ab}\otimes_{\Z}(\Z/p^{s}).
	$$
	Here we are using the notation $\Gamma^{\ab}=\Gamma/[\Gamma,\Gamma]$
	for the abelianization of $\Gamma$.
	Choose a filtration of open normal subgroups of $K_{p}$:
	$$
		K_{p}
		=
		G_{0}\supset G_{1}\supset \cdots,
	$$
	and let $\Gamma_{r}$ be the preimage of $G_{r}$ in $\Gamma$.
	By strong approximation, the map $\Gamma\to K_{p}$ induces
	isomorphisms on the quotients: $\Gamma_{t}/\Gamma_{r}=G_{t}/G_{r}$.
	By Emerton's formula, we have
	$$
		\tH_{\cpct}(K^{p},\Z_{p})
		=
		\limprojs\limdr H_{1}(\Gamma_{r},\Z/p^{s}).
	$$
	Note that we have a central extension of profinite groups:
	$$
		1 \to \Cong(\bG) \to \tilde K_{f,r} \to G_{r}K^{p} \to 1,
	$$
	where $\tilde K_{f,r}$ is the profinite completion of $\Gamma_{r}$.
	As a consequence of this we have
	$$
		\Gamma_{r}^{\ab}\otimes \Zps
		=
		\tilde K_{f,r}^{\ab}\otimes \Zps.
	$$
	Since $\Cong(\bG)$ is finite,
	we may choose $K_{f}$ small enough so that the central extension splits.
	It follows that
	$$
		H_{1}(\Gamma_{r}, \Zps)
		=
		(\Cong(\bG)\oplus (K^{p})^{\ab})\otimes\Zps
		\oplus
		(G_{r}^{\ab}\otimes \Zps).
	$$
	Let $\verl^{r}_{t}$ denote the transfer map
	from $H_{1}(\Gamma_{r},\Zps)$ to $H_{1}(\Gamma_{t},\Zps)$.
	For an element $\sigma\in (\Cong(\bG)\oplus (K^{p})^{\ab})\otimes\Zps$,
	we have
	$$
		\verl^{r}_{t}(\sigma)
		=
		[G_{r}:G_{t}]\sigma.
	$$
	It follows that such classes vanish in the
	limit over the levels $G_{r}$, and we have:
	$$
		\limdr
		H_{1}(\Gamma_{r},\Zps)
		=
		\limdr
		(G_{r}^{\ab}\otimes \Zps).
	$$	
	If $r$ is large enough then $G_{r}$ is a pro-$p$ group, and
	$G_{r}^{\ab}$ is always finite, so we have:
	$$
		\limdr
		H_{1}(\Gamma_{r},\Zps)
		=
		\limdr
		\Hom_{\Z}(H^{1}_{\cts}(G_{r},\Zps), \Q/\Z).
	$$
	Recall that by a theorem of Lazard
	 (Cor. 4.8 of Annexe 2.4 of \cite{serre-galoiscohomology}),
	 the groups $G_{r}$	are \emph{Poincar\'e duality groups at $p$}
	 of dimension $n=\dim \bG$.
	In particular, by Prop. 4.3 of Annexe 2.4 of \cite{serre-galoiscohomology},
	we have:
	$$
		\limdr
		\Hom_{\Z}(H^{q}_{\cts}(G_{r},\Z/p),\Q/\Z)
		=
		0
		\quad
		\hbox{for $q\ne \dim_{\Q}\bG$.}
	$$
	In particular, this holds for $q=1$.
	By induction on $s$, we deduce that
	$$
		\limdr
		\Hom_{\Z}(H^{1}_{\cts}(G_{r},\Z/p^{s}),\Q/\Z)
		=
		0.
	$$
	The result follows in the compactly supported case.
	To prove the result in the case of usual cohomology, we
	apply Corollary \ref{equality}.
	We need only check that $d\ge D(\bG,\Z_{p})+2$.
	Suppose $\bP=\bM\bA\bN$ is a proper parabolic subgroup of $\bG$.
	We have from the Iwasawa decomposition:
	$$
		X_{\bG}=X_{\bM}\bA(\R)^{\circ}\bN(\R).
	$$
	In particular since neither $\bA$ nor $\bN$ is trivial, we have
	$$
		d \ge \dim(X_{\bM})+2.
	$$
	We have $\cd_{\cpct}(\bM,\Z_{p})\le \dim(X_{\bM})$,
	so the result follows.
\end{proof}

In view of the previous results,
 the author is are lead to the following conjecture:

\begin{conjecture}
	Let $\bG/\Q$ be a reductive group whose centre has
	rational rank zero.
	Then
	$$
		\cd(\bG,\Z_{p}),
		\cd_{\cpct}(\bG,\Z_{p})
		\le
		\dim X_{\bG}-\rank_{\R}\bG.
	$$
\end{conjecture}

We have proved the conjecture for groups whose
real rank is equal to the rational rank (Theorem \ref{vanishingtheorem}),
and in particular for all split groups.
We have also proved the conjecture for groups of real rank $\le 1$
(Proposition \ref{topdimension}).
We have shown (Proposition \ref{nextdimensiondown})
 that for groups of real rank $2$,
 the conjecture follows from Serre's conjecture on congruence kernels.
We have shown (Theorem \ref{universalcover})
 that to prove the conjecture for a semi-simple group, it is sufficient to prove it for the
 universal cover.
We have also shown (Corollary \ref{leopoldt})
 that for a torus, the conjecture is equivalent to Leopoldt's conjecture.
Indeed if Leopoldt's conjecture is correct, then it would be sufficient
to prove the conjecture for semi-simple, groups.
Finally, we show that the usual and compactly supported cases of the conjecture
are equivalent:

\begin{proposition}
	\label{vanishingtheoremconj}
	Let  ``$-$'' denote either $\Z_{p}$ or $\Q_{p}$.
	The following statements are equivalent:
	\begin{enumerate}
		\item
		for every reductive group $\bG/\Q$ such that the maximal $\Q$-split torus
		 in the centre of $\bG$ is trivial,
		 $\cd(\bG,-)\le \dim(X_{\bG})-\rank_{\R}(\bG)$;
		\item
		for every reductive group $\bG/\Q$ such that the maximal $\Q$-split torus
		 in the centre of $\bG$ is trivial,
		 $\cd_{\cpct}(\bG,-)\le \dim(X_{\bG})-\rank_{\R}(\bG)$.
	\end{enumerate}
\end{proposition}

\begin{proof}
	One proves this by the same method as Theorem \ref{vanishingtheorem}.
\end{proof}

\end{document}